\documentclass[12pt]{amsart}
\usepackage{latexsym,amssymb,amsmath, amsfonts}
\usepackage{graphicx}
\usepackage{pifont}
\usepackage{mathrsfs}
\usepackage{epsfig,verbatim}
\usepackage{xcolor}

\numberwithin{equation}{section}

\newtheorem{theorem}{Theorem}[section]
\newtheorem{lemma}[theorem]{Lemma}

\newtheorem{proposition}[theorem]{Proposition}
\newtheorem{definition}[theorem]{Definition}

\newcommand\R{\mathbb{R}}

\newcommand\N{\mathbb{N}}

\newcommand\be{\begin{equation}}
\newcommand\ee{\end{equation}}
\newcommand\baa{\begin{array}}
\newcommand\eaa{\end{array}}
\newcommand\bea{\begin{eqnarray}}
\newcommand\eea{\end{eqnarray}}
\newcommand\beaa{\begin{eqnarray*}}
\newcommand\eeaa{\end{eqnarray*}}
\newcommand\bss{\begin{cases}}
\newcommand\ess{\end{cases}}
\newcommand\bR{\mathbb{R}}
\newcommand\bN{\mathcal{N}}
\newcommand\mU{{\mathcal{U}}}
\newcommand\mL{{\mathcal{L}}}

\newcommand\bNU{{\mathcal{N}_1[\phi_1]}}
\newcommand\bNV{{\mathcal{N}_2[\phi_2]}}
\newcommand\bNu{{\mathcal{N}_1[u]}}
\newcommand\bNv{{\mathcal{N}_2[v]}}
\newcommand\bNW{{\mathcal{N}_3[\phi_3]}}
\newcommand\bNw{{\mathcal{N}_3[w]}}

\newcommand\uphi{{\underline{\phi}}}
\newcommand\ophi{{\overline{\phi}}}

\newcommand\ld{\lambda}
\newcommand\ep{\varepsilon}
\newcommand\phitri{(\phi_1,\phi_2,\phi_3)}
\newcommand\estar{e^{-\ld^{**}_2z}}

\begin{document}

\title[nonlocal dispersal and climate change]
{Forced waves for a three-species predator-prey system with nonlocal dispersal in a shifting environment}

\author[J.-S. Guo]{Jong-Shenq Guo}
\address{Department of Mathematics, Tamkang University, Tamsui, New Taipei City 251301, Taiwan}
\email{jsguo@mail.tku.edu.tw}

\author[F. Hamel]{Fran{\c{c}}ois Hamel}
\address{Aix Marseille Univ, CNRS, I2M, Marseille, France}
\email{francois.hamel@univ-amu.fr}

\author[C.-C. Wu]{Chin-Chin Wu}
\address{Department of Applied Mathematics, National Chung Hsing University, Taichung 402, Taiwan }
\email{chin@email.nchu.edu.tw}

\thanks{Date: \today. Corresponding author: C.-C. Wu}

\thanks{This work was supported in part by the Ministry of Science and Technology of Taiwan under the grants 108-2115-M-032-006-MY3 (JSG) and 110-2115-M-005-001-MY2 (CCW), and by the CNRS-NCTS joint International Research Network ReaDiNet. This work has also received funding from Excellence Initiative of Aix-Marseille Universit\'e~-~A*MIDEX, a French ``Investissements d'Avenir'' programme, and from the ANR RESISTE (ANR-18-CE45-0019) project.}



\begin{abstract}
We consider a three-species predator-prey system involving two competing predators and one prey. The species diffuse with nonlocal dispersal kernels with possibly non-compact support, 
and they interact in a heterogeneous environment moving with a positive forced speed such that 
the environment is favorable to the prey in the absence of predators far ahead of the shifting boundary and it is unfavorable far behind. 
Such systems arise in the modeling of population dynamics under the effect of a shifting environment, such as climate change. 
We show on the one hand the existence of waves connecting the trivial state to the unique constant positive co-existence state for any value of the forced speed. 
On the other hand, we show the existence of critical positive speeds for the existence of waves connecting the trivial state to the states corresponding to the absence of one or two predators.
\end{abstract}

\maketitle


\section{Introduction}
\setcounter{equation}{0}

Nonlocal dispersal models arise in population dynamics to describe long-distance dispersal of individuals \cite{k03,klv96,Lutscher}. Recently, the effect of environmental heterogeneity has drawn a lot of attention in biological applications. In this paper, we consider the following predator-prey system with two weak competing predators and one prey:
\be\label{pp}\left\{\baa{rcll}
u_t(x,t) & = & d_1\bNu(x,t) & \!\!\!\!+\,r_1u(x,t)\,[-1-u(x,t)-kv(x,t)+aw(x,t)],\vspace{3pt}\\
v_t(x,t) & = & d_2\bNv(x,t) & \!\!\!\!+\,r_2v(x,t)\,[-1-hu(x,t)-v(x,t)+aw(x,t)],\vspace{3pt}\\
w_t(x,t) & = & d_3\bNw(x,t) & \!\!\!\!+\,r_3w(x,t)\,[\alpha(x-st)-bu(x,t)-bv(x,t)-w(x,t)],\eaa\right.
\ee
in which the nonnegative quantities $u=u(x,t)$ and $v=v(x,t)$ are the densities of two predators and $w=w(x,t)$ is the density of the single prey, for $x,t\in\R$.

Throughout the paper, the parameters $d_1,d_2,d_3,r_1,r_2,r_3,a,b,h,k$ are all positive, and $a,b,h,k$ are such that
\be\label{c1}
a>1,\quad 0<h,k<1,\quad 0<b<\frac{1}{2(a-1)}.
\ee

Meanwhile, $\bNu$, $\bNv$ and $\bNw$ formulate the spatial nonlocal dispersal and are defined by
$$\left\{\baa{lclcl}
\bNu(x,t) & := & (J_1\ast u)(x,t)-u(x,t) & = & \displaystyle\int_{\mathbb{R}}J_{1}(y)u(x-y,t)dy-u(x,t),\vspace{3pt}\\
\bNv(x,t) & := & (J_2\ast v)(x,t)-v(x,t)& = & \displaystyle\int_{\mathbb{R}}J_{2}(y)v(x-y,t)dy-v(x,t),\vspace{3pt}\\
\bNw(x,t) & := & (J_3\ast w)(x,t)-w(x,t) & = & \displaystyle\int_{\mathbb{R}}J_{3}(y)w(x-y,t)dy-w(x,t),\eaa\right.$$
in which, throughout the paper, the functions $J_{i}:\R\to\R$ ($i=1,2,3$) are probability kernel functions satisfying
\begin{enumerate}
\item[(J1)] $J_i$ is nonnegative, measurable with respect to the Lebesgue measure, $\int_{\mathbb{R}}J_i(y)dy=1$, and there are $\eta_i>0$ and $y^\pm_i\in\R$ such that $y^-_i<0<y^+_i$ and $J_i>0$ in $(y^-_i-\eta_i,y^-_i+\eta_i)$ and in $(y^+_i-\eta_i,y^+_i+\eta_i)$;
\item[(J2)] $\int_{\R}J_i(y)y\,dy=0$;
\item[(J3)] there exist $-\infty\le\tilde\lambda_i<0<\hat\lambda_i\le+\infty$ such that $I_i(\ld)<+\infty$ for all $\ld\in(\tilde{\lambda}_i,\hat{\ld}_i)$, and $I_i(\ld)\to+\infty$ as $\ld\downarrow\tilde{\lambda}_i$ and $\ld\uparrow\hat{\ld}_i$, where
$$I_i(\ld):=\int_\bR J_i(y)e^{\ld y}dy.$$
\end{enumerate}
Conditions (J1)-(J3) imply that each function $I_i$ is strictly convex and of class $C^\infty$ in $(\tilde\lambda_i,\hat\lambda_i)$, that $I_i(0)=1$, $I_i'(0)=0$, and $I_i(\lambda)=+\infty$ for all $\lambda\in(-\infty,\tilde{\lambda}_i]$ if $-\infty<\tilde\lambda_i$ (resp. for all $\lambda\in[\hat\lambda_i,+\infty)$ if $\hat\lambda_i<+\infty$). If $J_i$ has a compact support, then $\tilde{\lambda}_i=-\infty$ and $\hat\lambda_i=+\infty$. Conditions~(J1) and~(J3) imply that each function $y\mapsto J_i(y)y$ is in $L^1(\R)$ and, if $J_i$ is even, they necessarily yield~(J2), and $\tilde{\lambda}_i=-\hat\lambda_i$ in~(J3). We point out that the conditions (J1)-(J3) are satisfied in particular if $J_i$ is nonnegative, continuous, even, has a unit integral over $\R$ and if $I_i(\lambda)<+\infty$ for some $\lambda>0$. But the conditions (J1)-(J3) cover more general dispersal kernels $J_i$, in particular the kernels $J_i$ can be non-symmetric.

The function $\alpha$ in~\eqref{pp}, describing the heterogeneity, is assumed to be continuous in $\bR$ and the given positive constant $s$ denotes the environmental shifting speed. Throughout this paper we also impose the following conditions on $\alpha$:
\begin{enumerate}
\item[($\alpha$1)] $\alpha$ has limits $\alpha(\pm\infty)$ at $\pm\infty$, such that $-\infty<\alpha(-\infty)<0<\alpha(+\infty)<+\infty$, and $\alpha(z)\le\alpha(+\infty)$ for all $z\in\R$;
\item[($\alpha$2)] there exist $C>0$ and $\rho>0$ such that $\alpha(+\infty)-\alpha(z)\le Ce^{-\rho z}$ for all large $z$.
\end{enumerate}
Condition ($\alpha$1) means that the environment is favourable to the prey ahead of the shifting boundary $x=st$, then gradually deteriorates until it becomes hostile to the species far behind this shifting boundary. We point out that the function $\alpha$ is not assumed to be monotone.

Biologically, parameters $(d_1,d_2,d_3)$, $(h,k)$, $a$ and $b$ represent the diffusion coefficients, competition rates, conversion rate and predation rate, respectively. The assumption $0<h,k<1$ means that the two predators are weak competitors. Moreover, the negative net growth rate $-r_i$ ($i=1,2$) of each predator means that each predator cannot survive without the feeding of the prey. The terms $-r_1u(x,t)^2$ and $-r_2v(x,t)^2$ stand for the intra-specific competition inside each predator population. The smallness of the predation rate $b$ means that the predation has a relatively low impact on the prey, whereas the largeness of the conversion rate $a$ is related to the relatively large fitness of the predators in the presence of the prey. The intrinsic growth rate of prey is given by $r_3\alpha(x-st)$ which is temporal-spatial dependent and takes both positive and negative values. From the modeling point of view, the function $\alpha$ represents the shifting environment effect,
 such as climate change. Although this term does not appear in the equation of each predator, the changing effect actually affects indirectly both predators due to the fact that predators are fed by prey.

We are concerned with the existence of forced waves for \eqref{pp}. Namely, a traveling wave solution of \eqref{pp} with speed $s$ is a solution in the form
$$(u,v,w)(x,t)=(\phi_1,\phi_2,\phi_3)(z),\ z:=x-st,$$
with $C^1(\R)$ functions $\phi_1,\phi_2,\phi_3$. Then $(\phi_1,\phi_2,\phi_3)$ satisfies
\be\label{TWS2}
\begin{cases}
-s\phi_1'(z)=d_1\bNU(z)+r_1\phi_1(z)\,[-1-\phi_1(z)-k\phi_2(z)+a\phi_3(z)], & z\in\bR,\\
-s\phi_2'(z)=d_2\bNV(z)+r_2\phi_2(z)\,[-1-h\phi_1(z)-\phi_2(z)+a\phi_3(z)], & z\in\bR,\\
-s\phi_3'(z)=d_3\bNW(z)+r_3\phi_3(z)\,[\alpha(z)-b\phi_1(z)-b\phi_2(z)-\phi_3(z)], & z\in\bR,
\end{cases}
\ee
where
$$\bN_i[\phi_i](z):=\int_\bR J_i(y)\phi_i(z-y)dy-\phi_i(z),\; i=1,2,3.$$
Throughout the paper, by a solution of~\eqref{TWS2}, we always mean a triplet $(\phi_1,\phi_2,\phi_3)$ of $C^1(\R)$ nonnegative bounded functions. Since the environment is hostile to the species far behind the shifting boundary $x=st$, by the assumption on $\alpha$, it can be expected that all species go extinction eventually. Hence we impose the boundary condition
$$(\phi_1,\phi_2,\phi_3)(-\infty)=(0,0,0).$$
On the other hand, without loss of generality we may assume that
\be\label{alpha1}
\alpha(+\infty)=1.
\ee
Condition~\eqref{alpha1} is assumed in all main results (Theorems~\ref{th:e4}-\ref{th:c2}). Then the following states are the only possible constant and non-trivial limiting states $(\phi_1,\phi_2,\phi_3)(+\infty)$ at $z=+\infty$:
$$E_1:=(0,0,1),\ \ E_2:=(u_p,0,w_p),\ \ E_3:=(0,u_p,w_p),\ \ E_4:=(u^*,v^*,w^*),$$
where the real numbers $u_p,w_p,u^*,v^*,w^*$ are all positive and given by
$$\begin{cases}
\displaystyle u_p:=\frac{a-1}{ab+1},\; w_p:=\frac{b+1}{ab+1},\vspace{5pt}\\
\displaystyle w^*:=\frac{1+b\gamma}{1+ab\gamma},\;\gamma:=\frac{2-h-k}{1-hk}>1,\; v^*:=\frac{1-h}{1-hk}(aw^*-1),\; u^*:=\frac{1-k}{1-hk}(aw^*-1).\end{cases}$$
Biologically, the state $E_1$ corresponds to a saturated aboriginal prey living in the habitat and there are two invading alien predators; $E_2$ or $E_3$ is a pair of aboriginal co-existent predator-prey and an invading alien predator; lastly, $E_4$ is the positive co-existence state.

The study of forced waves has attracted a lot of attention recently. We refer the reader to, e.g., \cite{bd09,bf18,br08,br09,bg19,fpz21,flw16,h97,hz17,v15,zk13} for the case of scalar equations with local diffusion and to~\cite{co1,dsz20,lwz18,wz19,zz20} for the case of scalar equations with nonlocal dispersal. For two-species models, we refer the reader to~\cite{ywl19} for a cooperative model, \cite{bdd14,dll21,ww21,wwz19,ywz19} for competition models, and~\cite{choi21} for a predator-prey system with both classical diffusion and nonlocal dispersal. Recently, forced waves for a three-species predator-prey system were investigated in~\cite{choi22} for two competing preys and one predator with the classical diffusion. In~\cite{choi22}, the authors obtained both front and pulse types forced waves.

In the case of nonlocal dispersal, one usually assume that each kernel function is of compact support (so that $\tilde\lambda_i=-\infty$ and $\hat{\ld}_i=+\infty$ in (J3)). One of the motivations of this work is to remove the restriction on the compact support of kernel(s). This question has been left open in~\cite{choi21} for the two-species predator-prey system.

We now describe our main results as follows. We repeat that the condition~\eqref{c1} is assumed in all results, as are~(J1)-(J3) and ($\alpha 1$)-($\alpha 2$). First, the following theorem provides the existence of waves connecting $(0,0,0)$ and the co-existence state $E_4$, whatever the positive forced speed $s$ may be.

\begin{theorem}\label{th:e4}
In addition to~\eqref{c1} and~\eqref{alpha1}, suppose
\be\label{c2}
b<\min\left\{\frac{1-h}{2a},\frac{1-k}{2a}\right\}.
\ee
Then, for any $s>0$, there exists a positive\footnote{By positive, we mean that each component $\phi_i$ ($i=1,2,3$) is positive in $\R$.} solution $(\phi_1,\phi_2,\phi_3)$ of \eqref{TWS2} such that
$$(\phi_1,\phi_2,\phi_3)(-\infty)=(0,0,0)\ \hbox{ and }\ (\phi_1,\phi_2,\phi_3)(+\infty)=E_4.$$
\end{theorem}

Next, for waves connecting $(0,0,0)$ and the predator-free state $E_1=(0,0,1)$, we let
\be\label{defs*i}
s_i^*:=\inf_{\ld\in(0,\hat{\ld}_i)}Q_i(\ld),\ \ Q_i(\ld):=\frac{d_i[I_i(\ld)-1]+ r_i(a-1)}{\ld},\ \ \lambda\in(0,\hat\lambda_i),\ \ i=1,2.
\ee
From~\eqref{c1} and (J1)-(J3) and the comments after (J1)-(J3), each function $Q_i$ ($i=1,2$) is continuous and positive in $(0,\hat\lambda_i)$, $Q_i(\lambda)\to+\infty$ as $\lambda\downarrow0$ or $\lambda\uparrow\hat\lambda_i$ (this last property is immediate if $\hat\lambda_i<+\infty$, and it holds as well if $\hat\lambda_i=+\infty$ since $I'_i(\lambda)=\int_{\R}J_i(y)ye^{\lambda y}dy\to+\infty$ as $\lambda\to+\infty=\hat\lambda_i$ in that case). Hence, each $s^*_i$ is positive and the infimum in~\eqref{defs*i} is a minimum, and it is reached by a unique $\lambda^*_i\in(0,\hat\lambda_i)$ (since $I_i''$ is positive in $(0,\hat\lambda_i)$, and even in the whole interval of definition $(\tilde\lambda_i,\hat\lambda_i)$). Furthermore, $Q_i$ is decreasing in $(0,\lambda^*_i]$ and increasing in $[\lambda^*_i,\hat\lambda_i)$. Then we have

\begin{theorem}\label{th:e1}
Suppose~\eqref{c1} and~\eqref{alpha1}. If $s>\max\{s_1^{*},s^*_2\}$, then there exists a positive solution $(\phi_1,\phi_2,\phi_3)$ of \eqref{TWS2} such that
\be\label{bc-1}
(\phi_1,\phi_2,\phi_3)(-\infty)=(0,0,0)\ \hbox{ and }\ (\phi_1,\phi_2,\phi_3)(+\infty)=E_1.
\ee
Moreover, when $(\tilde\lambda_i,\hat\lambda_i)=(-\infty,+\infty)$ for some $i\in\{1,2\}$, then a positive solution $\phitri$ of~\eqref{TWS2} and~\eqref{bc-1} exists only if $s\ge s_i^*$.
\end{theorem}

For the waves connecting $(0,0,0)$ and the mixed state $E_2=(u_p,0,w_p)$, we let
\be\label{defbeta2}
\beta_2:=-1-hu_p+aw_p=\frac{(a-1)(1-h)}{ab+1}>0
\ee
and
\be\label{defs**2}
s_2^{**}:=\inf_{\ld\in(0,\hat{\ld}_2)}R_2(\ld),\ \ R_2(\ld):=\frac{d_2[I_2(\ld)-1]+r_2\beta_2}{\ld},\ \ \lambda\in(0,\hat\lambda_2).
\ee
As for $Q_i$ in~\eqref{defs*i}, the function $R_2$ is continuous and positive in $(0,\hat\lambda_2)$, $R_2(\lambda)\to+\infty$ as $\lambda\downarrow0$ or $\lambda\uparrow\hat\lambda_2$,~$s^{**}_2$ is positive and the infimum in the definition of $s_2^{**}$ is a minimum, reached by a unique $\lambda_2^{**}\in(0,\hat\lambda_2)$. Furthermore, $R_2$ is decreasing in $(0,\lambda^{**}_2]$ and increasing in $[\lambda^{**}_2,\hat\lambda_2)$. Then we have

\begin{theorem}\label{th:e2}
In addition to~\eqref{c1} and~\eqref{alpha1}, suppose
\be\label{de2}
\max\{d_1,d_3\}\le d_2,\ \ J_1=J_2=J_3\hbox{ in $\bR$,}
\ee
and
\be\label{re2}
r_1[1+k(a-1)]\le r_2\beta_2.
\ee
If $s>s_2^{**}$ and $s\ge R_2(\rho)$, with $\rho>0$ as in~{\rm{($\alpha 2$)}}, then there exists a positive solution $(\phi_1,\phi_2,\phi_3)$ of \eqref{TWS2} such that
\be\label{bc-2}
(\phi_1,\phi_2,\phi_3)(-\infty)=(0,0,0)\ \hbox{ and }\ (\phi_1,\phi_2,\phi_3)(+\infty)=E_2.
\ee
Moreover, when $(\tilde\lambda_2,\hat\lambda_2)=(-\infty,+\infty)$, a positive solution $\phitri$ of \eqref{TWS2} and \eqref{bc-2} exists only if $s\ge s_2^{**}$.
\end{theorem}

Notice that, if condition ($\alpha 2$) is satisfied for a certain $\rho>0$, so is it for every $\rho'\in(0,\rho]$. Therefore, in Theorem~\ref{th:e2}, one can always assume without loss of generality that $\rho\in(0,\hat\lambda_2)$ and that $R_2(\rho)$ is a real number. We also point out that, if $\alpha(z)=\alpha(+\infty)$ for all $z$ large enough, then ($\alpha 2$) is satisfied for all $\rho>0$, hence $\rho$ can be chosen so that $R_2(\rho)=s_2^{**}$ and the condition $s\ge R_2(\rho)$ in the statement can then be dropped.

By exchanging the roles of $u$ and $v$, a similar theorem to Theorem~\ref{th:e2}, on forced waves connecting the trivial state and $E_3$, can be derived. We omit it here.

One should notice that the above theorems hold under the assumptions (J1)-(J3) on kernels $J_i$, $i=1,2,3$, and conditions ($\alpha$1)-($\alpha$2) on $\alpha$. In particular, the kernel functions $J_i$, $i=1,2,3$, are not assumed to be compactly supported or symmetric. It is worth to mention that the method of this work can also be applied to the 2-species system studied in~\cite{choi21} and that the compact support assumption and the symmetry of the kernels can then be removed.

Next, we turn to the derivation of waves with critical speeds. We divide our main results into three cases. The following two theorems show, under some conditions, the existence of forced waves connecting the trivial state to the predator-free state $E_1=(0,0,1)$, for the minimal possible shifting speed. We recall that $s^*_1$ and $s^*_2$ are defined in~\eqref{defs*i}.

\begin{theorem}\label{th:c1}
In addition to~\eqref{c1} and~\eqref{alpha1}, suppose that $s^*_1=s^*_2$, and that $J_1$ and $J_2$ have compact supports. Then a positive solution of~\eqref{TWS2} with~\eqref{bc-1} exists for $s=s_1^*=s_2^*$.
\end{theorem}

Note that the condition $s^*_1=s^*_2$ is fulfilled especially if $d_1=d_2$, $r_1=r_2$, and $J_1=J_2$ in $\R$.

\begin{theorem}\label{th:c12}
In addition to~\eqref{c1} and~\eqref{alpha1}, suppose that $s^*_1>s^*_2$ and that $J_1$ has a compact support. Then a positive solution of~\eqref{TWS2} with~\eqref{bc-1} exists for $s=s_1^*$.
\end{theorem}

Note that, from the comments after (J1)-(J3), the condition $s^*_1>s^*_2$ is fulfilled if $d_1\ge d_2$, $r_1\ge r_2$, $(d_1,r_1)\neq(d_2,r_2)$, and $J_1=J_2$ in $\R$. A similar result to Theorem~\ref{th:c12} also holds if $s^*_2>s^*_1$ and if $J_2$ has a compact support (namely, in that case, a forced wave connecting the trivial state and $E_1$ with the critical speed $s=s^*_2$ exists).

Lastly, the next theorem shows, under some conditions, the existence of forced waves connecting the trivial state to the mixed state $E_2=(u_p,0,w_p)$, for the minimal possible shifting speed.

\begin{theorem}\label{th:c2}
In addition to~\eqref{c1} and~\eqref{alpha1}, suppose that $d_1=d_2=d_3$, that $J:=J_1=J_2=J_3$ in $\R$, and that $J$ has a compact support. We also assume that~\eqref{re2} holds and that $\rho\ge\lambda^{**}_2$, where $\rho>0$ is as in~$($$\alpha 2$$)$ and $\lambda^{**}_2>0$ denotes the unique minimum of $R_2$ in $(0,\hat\lambda_2)=(0,+\infty)$. Then a positive solution of~\eqref{TWS2} with~\eqref{bc-2} exists for $s=s_2^{**}$.
\end{theorem}

By exchanging the roles of $u$ and $v$, a similar theorem to Theorem~\ref{th:c2} on forced waves connecting the trivial state and $E_3$ with critical speed can be derived. We omit it here.

Predator-prey systems such as~\eqref{pp} are neither cooperative nor competitive, and the maximum principle does not hold for these systems. Furthermore, due to the heterogeneity of the function $\alpha$, the systems~\eqref{pp} and~\eqref{TWS2} are not invariant by translation with respect to the spatial variable, and new difficulties then arise in comparison with the existence of traveling waves for homogeneous predator-prey systems (we refer to, e.g.,~\cite{cgg21,cg21,cgy17,g21,gnow20} for the existence of waves for various homogeneous predator-prey systems with local diffusion, and to~\cite{dggs21,dgm19,dglp19} for the study of spreading speeds for such homogeneous systems with local or nonlocal dispersal).

With respect to the paper~\cite{choi21} on systems with both classical diffusion and nonlocal dispersal, we introduce in the present paper, in order to remove the assumption on compact support, some new ideas based on the Schauder fixed-point theorem, the suitable definition of upper and lower solutions, and translation strategies described in Section~2. The existence of forced waves connecting co-existence or predator-free states to the trivial state then relies on the construction of suitable upper and lower solutions, which is quite intricate for this three-species system.

\vskip 3mm
\noindent{\bf{Organization of the paper.}} Section~2 presents a general framework for the existence of waves, based on fixed point methods and translation strategies. Theorems~\ref{th:e4}-\ref{th:e2} on the existence of super-critical waves connecting $(0,0,0)$ and $E_4$, $E_1$ or $E_2$ are proved in Sections~3-5. Section~6 is devoted to the proof of Theorems~\ref{th:c1}-\ref{th:c2} on the existence of forced waves with critical speeds.


\section{Preliminaries}
\setcounter{equation}{0}

First, we introduce the {notion of} generalized upper-lower solutions of \eqref{TWS2}.

\begin{definition}\label{lus}
Continuous functions $(\overline{\phi}_1,\overline{\phi}_2,\overline{\phi}_3)$ and $(\underline{\phi}_1,\underline{\phi}_2,\underline{\phi}_3)$
are called a pair of upper and lower solutions of \eqref{TWS2} if $\overline{\phi_i}\ge\underline{\phi_i}$ in $\R$, $i=1,2,3$, and the following inequalities
\begin{eqnarray}
&\mathcal{U}_1(z):=d_{1} \bN_1[\overline{\phi}_1](z)+s\overline{\phi}_1^{\prime}(z)+r_{1}\overline{\phi}_1(z)\,[-1-\overline{\phi}_1(z)-k\underline{\phi}_2(z)+a\overline{\phi}_3(z)] & \!\!\!\le 0,  \label{u1} \\
&\mU_2(z):=d_{2} \bN_2[\overline{\phi}_2](z)+s\overline{\phi}_2^{\prime}(z)+r_{2}\overline{\phi}_2(z)\,[-1-h\underline{\phi}_1(z)-\overline{\phi}_2(z)+a\overline{\phi}_3(z)] & \!\!\!\le 0,  \label{u2} \\
&\mU_3(z):=d_{3} \bN_3[\overline{\phi}_3](z)+s\overline{\phi}_3^{\prime}(z)+r_{3}\overline{\phi}_3(z)\,[\alpha(z)-b\underline{\phi}_1(z)-b\underline{\phi}_2(z)-\overline{\phi}_3(z)] & \!\!\!\le 0,  \label{u3} \\
&\mL_1(z):= d_{1} \bN_1[\underline{\phi}_1](z)+s\underline{\phi}_1^{\prime}(z)+r_{1}\underline{\phi}_1(z)\,[-1-\underline{\phi}_1(z)-k\overline{\phi}_2(z)+a\underline{\phi}_3(z)] & \!\!\!\ge 0,  \label{l1} \\
&\mL_2(z):= d_{2} \bN_2[\underline{\phi}_2](z)+s\underline{\phi}_2^{\prime}(z)+r_{2}\underline{\phi}_2(z)\,[-1-h\overline{\phi}_1(z)-\underline{\phi}_2(z)+a\underline{\phi}_3(z)] & \!\!\!\ge 0,  \label{l2} \\
&\mL_3(z):= d_{3} \bN_3[\underline{\phi}_3](z)+s\underline{\phi}_3^{\prime}(z)+r_{3}\underline{\phi}_3(z)\,[\alpha(z)-b\overline{\phi}_1(z)-b\overline{\phi}_2(z)-\underline{\phi}_3(z)] & \!\!\!\ge 0  \label{l3}
\end{eqnarray}
hold for all $z \in \mathbb{R}\setminus E$ for some finite subset $E$ $($possibly empty$)$ of $\mathbb{R}$ $($the functions~$\underline{\phi}_i$ and~$\overline{\phi}_i$ are assumed to be of class $C^1$ in, at least, $\R\setminus E$$)$.
\end{definition}

Next, from the Schauder fixed-point theorem, we can derive the existence of wave profiles as follows. Note that our proof below actually only requires the nonnegativity and integrability of the kernels $J_i$.

\begin{lemma}\label{luslem}
Let $s>0$ be given. Let $(\overline{\phi}_1,\overline{\phi}_2,\overline{\phi}_3)$ and $(\underline{\phi}_1,\underline{\phi}_2,\underline{\phi}_3)$ be a pair of nonnegative bounded upper and lower solutions of \eqref{TWS2}. Then \eqref{TWS2} admits a solution $(\phi_1,\phi_2,\phi_3)$ such that $\underline{\phi}_i\le \phi_i\le \overline{\phi}_i$ in $\R$ for all $i=1,2,3$.
\end{lemma}

\begin{proof}
First of all, we denote
$$M:=\max\big\{\|\alpha\|_{L^\infty(\R)},\|\overline{\phi}_1\|_{L^\infty(\R)},\|\overline{\phi}_2\|_{L^\infty(\R)},\|\overline{\phi}_3\|_{L^\infty(\R)}\big\}$$
and we define a positive constant $\beta$ as follows:
\be\label{beta-con}
\beta:=\max\{\sigma_1,\sigma_2,\sigma_3\}>0,
\ee
where
\be\label{beta-con2}
\sigma_1:=d_1 + r_1(2M+kM+1),\ \sigma_2:=d_2+r_2(2M+hM+1),\ \sigma_3:=d_3+r_3M(b+3).
\ee
We also fix an arbitrary positive continuous function $W:\R\to\R$ (a weight function) such that $W(z)\to0$ as $z\to\pm\infty$ and we consider the Banach space
$$X:=\big\{(\phi_1,\phi_2,\phi_3):\phi_i\in C^0(\R)\cap L^\infty(\R),\ i=1,2,3\big\}$$
endowed with the norm
\be\label{normX}
\|(\phi_1,\phi_2,\phi_3)\|_X:=\max\big\{\|\phi_1W\|_{L^\infty(\R)},\|\phi_2W\|_{L^\infty(\R)},\|\phi_2W\|_{L^\infty(\R)}\big\}.
\ee
We then set
$$\Gamma:=\big\{(\phi_1,\phi_2,\phi_3):\phi_i\in C^0(\R),\ 0\le\underline{\phi}_i\le\phi_i\le\overline{\phi}_i\hbox{ in }\R,\ i=1,2,3\big\},$$
which is a non-empty convex closed bounded subset of $(X,\|\cdot\|_X)$.

Then we transform the differential system~\eqref{TWS2} to an integral operator (as in e.g. \cite{dglp19,gnow20}), and we will apply a fixed point theorem in $\Gamma$. To do so, we set, for $\Phi:=(\phi_1,\phi_2,\phi_3)\in\Gamma$,
$$\left\{\baa{ll}
\!F_1[\Phi](y):=\beta\phi_1(y) + d_1\mathcal{N}_1[\phi_1](y)+r_1\phi_1(y)\,[-1-\phi_1(y)-k\phi_2(y)+a\phi_3(y)], & \!y\in\R,\vspace{3pt}\\
\!F_2[\Phi](y):=\beta\phi_2(y) + d_2\mathcal{N}_2[\phi_2](y)+r_2\phi_2(y)\,[-1-h\phi_1(y)-\phi_2(y)+a\phi_3(y)], & \!y\in\R,\vspace{3pt}\\
\!F_3[\Phi](y):=\beta\phi_3(y) + d_3\mathcal{N}_3[\phi_3](y)+r_3\phi_3(y)\,[\alpha(y)-b\phi_1(y)-b\phi_2(y)-\phi_3(y)], & \!y\in\R,\eaa\right.$$
and then
$$P[\Phi]:=(P_1[\Phi],P_2[\Phi],P_3[\Phi]),$$
where
$$P_i[\Phi](z):=\frac{e^{\beta z/s}}{s}\!\int_z^{+\infty}\!\!e^{-\beta y/s}F_i[\Phi](y)\,dy=\frac{1}{s}\!\int_0^{+\infty}\!\!e^{-\beta u/s}F_i[\Phi](z+u)\,du,\ z\in\R,\ i=1,2,3.$$
For each $\Phi\in\Gamma$, the functions $F_i[\Phi]$ are bounded and continuous in $\R$ and then so are the functions $P_i[\Phi]$, for $i=1,2,3$. Furthermore, the functions $P_i[\Phi]$ are actually of class $C^1(\R)$. It is straightforward to check that a fixed point $\Phi=(\phi_1,\phi_2,\phi_3)$ of $P$ in $\Gamma$ is a $C^1(\R)$ solution of~\eqref{TWS2} such that $\underline{\phi}_i\le\phi_i\le\overline{\phi}_i$ in $\R$ for $i=1,2,3$.

To derive a fixed point of $P$ in $\Gamma$, we first show that $P$ maps $\Gamma$ into $\Gamma$. To this aim, we let $\Phi\in\Gamma$. Then, using $\beta\ge\sigma_1$ in~\eqref{beta-con}-\eqref{beta-con2} together with the definition of $\Gamma$, we get that
$$F_1[\Phi](z)\ge F_1[(\underline{\phi}_1,\overline{\phi}_2,\underline{\phi}_3)](z),$$
hence $P_1[\Phi](z)\ge P_1[(\underline{\phi}_1,\overline{\phi}_2,\underline{\phi}_3)](z)$, for all $z\in\R$. On the other hand, by the definition of upper-lower solutions, especially~\eqref{l1} and the finiteness of the set $E$ in Definition~\ref{lus}, we have
$$\baa{rcl}
P_1[(\underline{\phi}_1,\overline{\phi}_2,\underline{\phi}_3)](z) & = & \displaystyle\frac{1}{s}\!\int_z^{+\infty}\!\!e^{\beta(z-y)/s}F_1[(\underline{\phi}_1,\overline{\phi}_2,\underline{\phi}_3)](y)\,dy\vspace{3pt}\\
& = & \displaystyle\mathop{\lim}_{\varepsilon\downarrow0}\int_{[z,+\infty)\setminus\cup_{t\in E}(t-\varepsilon,t+\varepsilon)}e^{\beta(z-y)/s}\frac{F_1[(\underline{\phi}_1,\overline{\phi}_2,\underline{\phi}_3)](y)}{s}\,dy\vspace{3pt}\\
& \ge & \displaystyle\mathop{\liminf}_{\varepsilon\downarrow0}\int_{[z,+\infty)\setminus\cup_{t\in E}(t-\varepsilon,t+\varepsilon)}e^{\beta(z-y)/s}\Big(-\underline{\phi}_1'(y)+\frac{\beta}{s}\underline{\phi}_1(y)\Big)\,dy=\underline{\phi}_1(z)\eaa$$
for all $z\in\R$, since $\underline{\phi}_1$ is continuous in $\R$. Hence, $P_1[\Phi](z)\ge\underline{\phi}_1(z)$ for all $z\in\R$. Similarly, we can derive
$$\left\{\baa{ll}
P_1[\Phi]\le P_1[(\overline{\phi}_1,\underline{\phi}_2,\overline{\phi}_3)]\le\overline{\phi}_1 & \hbox{in }\R,\vspace{3pt}\\
\underline{\phi}_2\le P_2[(\overline{\phi}_1,\underline{\phi}_2,\underline{\phi}_3)]\le P_2[\Phi]\le P_2[(\underline{\phi}_1,\overline{\phi}_2,\overline{\phi}_3)]\le\overline{\phi}_2 & \hbox{in }\R,\vspace{3pt}\\
\underline{\phi}_3\le P_3[(\overline{\phi}_1,\overline{\phi}_2,\underline{\phi}_3)]\le P_3[\Phi]\le P_3[(\underline{\phi}_1,\underline{\phi}_2,\overline{\phi}_3)]\le\overline{\phi}_3 & \hbox{in }\R,\eaa\right.$$
by the choice of $\beta$ in~\eqref{beta-con}-\eqref{beta-con2} and the definition of upper-lower solutions. Hence, we finally obtain that $P(\Gamma)\subset\Gamma$.

Now, we observe that, from the definition of $\Gamma$ and the boundedness of the functions $\underline{\phi}_i$ and $\overline{\phi}_i$ ($i=1,2,3$), there is a positive constant $\overline{M}$ such that $|F_i[\Phi](z)|\le\overline{M}$ for all $\Phi\in\Gamma$, $z\in\R$ and $i=1,2,3$, hence $|P_i[\Phi](z)|\le\overline{M}/\beta$ and $|P_i[\Phi]'(z)|=|\beta P_i[\Phi](z)-F_i[\Phi](z)|/s\le2\overline{M}/s$ for all $\Phi\in\Gamma$, $z\in\R$ and $i=1,2,3$. Therefore, from Arzel\`a-Ascoli theorem (applied in each compact interval $[-m,m]$ with $m\in\mathbb{N}$), together with a diagonal extraction process and the definition ~\eqref{normX} of the norm $\|\cdot\|_X$ in $X$, it follows that each sequence in $P(\Gamma)$ has a converging subsequence in $(X,\|\cdot\|_X)$ (the limit then belongs to $\Gamma$ from the previous paragraph and the closedness of $\Gamma$).

Furthermore, for any $\Phi\in\Gamma$ and any sequence $(\Phi_n)_{n\in\mathbb{N}}$ in $\Gamma$ such that $\|\Phi_n-\Phi\|_X\to0$ as $n\to+\infty$, one has in particular $\Phi_n(z)\to\Phi(z)$ as $n\to+\infty$ in $\R^3$ for every $z\in\R$ and, together with the boundedness of the sequence $(\Phi_n)_{n\in\N}$ in $L^\infty(\R)^3$ and the Lebesgue dominated convergence theorem, one gets that $F_i[\Phi_n](z)\to F_i[\Phi](z)$ as $n\to+\infty$ for every $z\in\N$ and $i=1,2,3$. With the boundedness of the sequence $(F_1[\Phi_n],F_2[\Phi_n],F_3[\Phi_n])_{n\in\N}$ in $L^\infty(\R)^3$, one infers that $P[\Phi_n](z)\to P[\Phi](z)$ as $n\to+\infty$ in $\R^3$ for every $z\in\R$. Therefore, each converging subsequence of the sequence $(P[\Phi_n])_{n\in\N}$ in $(X,\|\cdot \|_X)$ must converge to the unique possible limit $P[\Phi]$, and finally $P[\Phi_n]\to P[\Phi]$ in $X$ as $n\to+\infty$.

We conclude from the previous two paragraphs that $P:\Gamma\to\Gamma$ is completely continuous. It thus follows from the Schauder fixed-point theorem that $P$ has a fixed point  $\Phi=(\phi_1,\phi_2,\phi_3)$ in $\Gamma$, which is then a $C^1(\R)$ solution of~\eqref{TWS2} such that $\underline{\phi}_i\le\phi_i\le\overline{\phi}_i$ in $\R$ for $i=1,2,3$. This completes the proof of the lemma.
\end{proof}

Moreover, due to the negativity of $\alpha(-\infty)$, we have the following universal result for the left-hand tail limit. Since its proof is almost the same as that of \cite[(3.19)]{choi21} (see also~\cite{W}), we omit it here safely.

\begin{proposition}\label{prop:plus}
It holds $(\phi_1,\phi_2,\phi_3)(-\infty)=(0,0,0)$ for any nonnegative bounded solution $(\phi_1,\phi_2,\phi_3)$ of~\eqref{TWS2}.
\end{proposition}

Note that any nonnegative nontrivial (in the sense that each component is not identically~$0$ in $\R$) solution $(\phi_1,\phi_2,\phi_3)$ of \eqref{TWS2} must be positive (in the sense that $\phi_i>0$ in $\bR$ for $i=1,2,3$), by a property similar to the strong maximum principle: indeed, if $\phi_i(z_0)=0$ for some $i\in\{1,2,3\}$ and $z_0\in\R$, then $\phi'_i(z_0)=0$, and $(J_i*\phi_i)(z_0)=0$ from the equation, hence $\phi_i=0$ in $[z_0-y^\pm_i-\eta_i,z_0-y^\pm_i+\eta_i]$ with $\eta_i>0$ and $-\infty<y^-_i<0<y^+_i<+\infty$ as in~(J1); one would then get that $\phi_i=0$ in $[z_0-ky^-_i-k\eta_i,z_0-ky^-_i+k\eta_i]$ for all $k\in\mathbb{N}$ by an immediate induction; by choosing positive integers $n$ and $k$ such that $z_0+ny^+_i\in[z_0-ky^-_i-k\eta_i,z_0-ky^-_i+k\eta_i]$, it follows that $\phi_i(z_0+ny^+_i)=0$, and then $\phi_i=0$ in $[z_0+(n-1)y^+_i-\eta_i,z_0+(n-1)y^+_i+\eta_i]$ and also $\phi_i=0$ in $[z_0-n\eta_i,z_0+n\eta_i]$ by induction; finally the non-empty set $\{z\in\R:\phi
 _i(z)=0\}$ would be both open and closed (by continuity of $\phi_i$), hence $\phi_i\equiv0$ in $\R$, a contradiction.

It is interesting to note from the boundedness of $\alpha$ (since it is continuous and has finite limits at $\pm\infty$) and~($\alpha2$) that $\alpha(+\infty)-\alpha(z)\le Ce^{-\rho z}$ for all $z\in\bR$, if we choose the constant~$C$ larger. Then for any positive constant $A$ it holds
$$\alpha(z+A)\ge\alpha(+\infty)-\ep e^{-\rho z}\hbox{ for all }z\in\bR,\hbox{ with }\ep:=Ce^{-\rho A}.$$
Observe that $\ep>0$ can be made as small as we want, if we choose $A$ large enough. We further note that $\phitri(\cdot+A)$ is a solution of~\eqref{TWS2} if and only if $\phitri$ is a solution of~\eqref{TWS2} with $\alpha$ replaced by $\alpha(\cdot+A)$. Therefore, in the sequel, (up to a translation) condition~($\alpha$2) can be rephrased without loss of generality as
\be\label{a2}
\alpha(z)\ge\alpha(+\infty)-\ep e^{-\rho z}\hbox{ for all }z\in\bR,
\ee
for a suitable choice of $\ep>0$ as small as we need.

Now, we consider the following scalar equation
\be\label{scalar}
-s\phi'(z)=d\bN[\phi](z)+r\phi(z)\,[\alpha(z)-\phi(z)],\; z\in\bR,
\ee
where parameters $s,d,r$ are positive constants, and
$$\bN[\phi](z):=\int_\bR J(y)\phi(z-y)dy-\phi(z)$$
in which the kernel $J$ is assumed to satisfy (J1)-(J3) with some constants $\eta>0$, $-\infty<y^-<0<y^+<+\infty$ in~(J1), and $-\infty\le\tilde\lambda<0<\hat{\ld}\le+\infty$ in~(J3), and the continuous heterogeneous function $\alpha$ is assumed to satisfy ($\alpha 1$)-($\alpha 2$), but its limit at $+\infty$ is not assumed to be equal to $1$ in general, that is,~\eqref{alpha1} is not assumed here. Recall from \cite[Theorem~4.5]{lwz18} that a positive nondecreasing solution $\phi$ of \eqref{scalar} satisfying
\be\label{1d-bc}
\phi(-\infty)=0\ \hbox{ and }\ \phi(+\infty)=\alpha(+\infty),
\ee
is known to exist when $J$ is assumed to be even and compactly supported and when $\alpha$ is assumed to be nondecreasing. Without the symmetry of $J$ and the assumption of compact support of $J$, and without the monotonicity of $\alpha$, we still get the existence of a forced wave for~\eqref{scalar}-\eqref{1d-bc} (which may nevertheless not be monotone in general):

\begin{proposition}\label{prop:key}
Under assumptions {\rm (J1)-(J3)} on $J$, with some constants $\eta>0$, $-\infty<y^-<0<y^+<+\infty$ in~{\rm{(J1)}}, and $-\infty\le\tilde\lambda<0<\hat{\ld}\le+\infty$ in~{\rm{(J3)}}, and under conditions~{\rm ($\alpha$1)-($\alpha$2)} on $\alpha$, problem~\eqref{scalar}-\eqref{1d-bc} with $s>0$ admits a positive $C^1(\R)$ solution $\phi$ such that $\alpha(+\infty)\ge\phi(z)\ge\alpha(+\infty)-Be^{-\ld_0 z}$ for all $z\in\R$, for some $\ld_0>0$ and $B>0$.
\end{proposition}

\begin{proof} First, we give a pair of super-sub-solutions to \eqref{scalar} as follows. It is clear that the constant function $\ophi(z)\equiv\alpha(+\infty)$ is a super-solution of \eqref{scalar}, in the sense that
$$d\bN[\ophi](z)+s\ophi'(z)+r\ophi(z)\,[\alpha(z)-\ophi(z)]\le0\ \hbox{ for all $z\in\R$}.$$
For the sub-solution, we consider the nonnegative bounded function
$$\uphi(z):=\max\{\alpha(+\infty)-e^{-\ld_0 z},0\}$$
for some small positive constant $\ld_0\in(0,\rho)$, where $\rho>0$ is the positive constant appearing in~($\alpha 2$). To choose $\ld_0$, we let
$$g(\ld):=d[I(\ld)-1]-s\ld,\ \ I(\ld):=\int_\bR J(y)e^{\ld y}dy,\ \ \lambda\in[0,\hat\lambda).$$
Note that $g(0)=0$ and $g'(0)=-s$, since
$$g'(\ld)=d\int_\bR yJ(y)e^{\ld y}dy - s\;\mbox{ and }\, \int_\bR yJ(y)dy=0$$
due to the assumptions~(J1)-(J3). Hence there is a $\ld_0\in(0,\min\{\rho,\hat\lambda\})$ such that $g(\ld_0)<0$. Now, we choose $\ep>0$ small enough such that
\be\label{choiceeps}
g(\ld_0)+r\ep\alpha(+\infty)^{\rho/\lambda_0}=g(\ld_0)+r\ep e^{(\rho/\lambda_0)\ln\alpha(+\infty)}<0
\ee
and \eqref{a2} holds for this $\ep$.
Then for $z>-\lambda_0^{-1}\ln\alpha(+\infty)=:z_0$ we have $\uphi(z)=\alpha(+\infty)-e^{-\lambda_0z}>0$ and, since $\uphi(x)\ge\alpha(+\infty)-e^{-\lambda_0x}$ for all $x\in\R$, we get that
\beaa
&\!\!\!\!&d\bN[\uphi](z)+s\uphi'(z)+r\uphi(z)[\alpha(z)-\uphi(z)]\\
& \!\!\ge\!\! & d\int_{\R}J(y)(\alpha(+\infty)\!-\!e^{-\lambda_0(z-y)})dy-d(\alpha(+\infty)\!-\!e^{-\lambda_0z})+s\lambda_0e^{-\lambda_0z}+r\uphi(z)[\alpha(z)\!-\!\uphi(z)]\\
&\!\!=\!\!&-e^{-\ld_0 z}\{d [I(\ld_0)-1]-s\ld_0\}+r\,(\alpha(+\infty)-e^{-\ld_0 z})\,(\alpha(z)-\alpha(+\infty)+e^{-\ld_0 z})\\
&\!\!\ge\!\!&-e^{-\ld_0 z}\{d [I(\ld_0)-1]-s\ld_0\}-r\ep(\alpha(+\infty)-e^{-\ld_0 z})e^{-\rho z}\\
&\!\!\ge\!\!&-e^{-\ld_0 z}\{d [I(\ld_0)\!-\!1]\!-\!s\ld_0\!+\!r\ep\alpha(+\infty)e^{-(\rho-\ld_0)z}\}\!\ge\!-e^{-\ld_0 z}[g(\ld_0)\!+\!r\ep e^{(\rho/\lambda_0)\ln\alpha(+\infty)}]\!>\!0,
\eeaa
using~\eqref{a2} and~\eqref{choiceeps}, together with $\ld_0<\rho$ and $z>z_0=-\lambda_0^{-1}\ln\alpha(+\infty)$. Furthermore, for $z<z_0$, we have $\uphi=0$ in a neighborhood of $z$ and $d\bN[\uphi](z)+s\uphi'(z)+r\uphi(z)[\alpha(z)-\uphi(z)]=d(J*\uphi)(z)\ge0$. Hence $\uphi$ is a sub-solution of \eqref{scalar}, in a sense similar to Definition~\ref{lus}, that is, it is continuous in $\R$, of class $C^1$ in $\R\setminus\{z_0\}$, and it satisfies $d\bN[\uphi](z)+s\uphi'(z)+r\uphi(z)[\alpha(z)-\uphi(z)]\ge0$ for all $z\in\R\setminus\{z_0\}$.

Next, since $0\le\underline{\phi}\le\overline{\phi}=\alpha(+\infty)$, it follows with the same arguments as in the proof of Lemma~\ref{luslem} that there is a $C^1(\R)$ solution $\phi$ of~\eqref{scalar} such that $0\le\uphi\le\phi\le\ophi=\alpha(+\infty)$ in $\bR$. Hence, $\phi(+\infty)=\alpha(+\infty)$, since $\uphi(+\infty)=\ophi(+\infty)=\alpha(+\infty)$. Furthermore, from the definition of $\uphi$ and the boundedness of $\phi$, there is $B\ge1$ such that $\phi(z)\ge\alpha(+\infty)-Be^{-\lambda_0z}$ for all $z\in\R$. Finally, $\phi$ is positive due to the strong maximum principle (as for the proof of the positivity of $(\phi_1,\phi_2,\phi_3)$ after the statement of Proposition~\ref{prop:plus}) and the property $\phi(-\infty)=0$ can be proved in the same way as that in~\cite[(3.19)]{choi21} (see also~\cite{W}), since $\alpha(-\infty)<0$. We omit it here. This completes the proof of the proposition.
\end{proof}

We remark here that the solution $\phi$ of \eqref{scalar}-\eqref{1d-bc} obtained in Proposition~\ref{prop:key} also satisfies condition of the type~($\alpha$2) by construction, since it is trapped between $\uphi(z)=\max\{\phi(+\infty)-e^{-\ld_0 z},0\}$ and $\ophi(z)=\phi(+\infty)$.


\section{Proof of Theorem~\ref{th:e4}}
\setcounter{equation}{0}

Throughout this section, we suppose~\eqref{c1} and~\eqref{alpha1}-\eqref{c2}. The proof of Theorem~\ref{th:e4} is similar to that of \cite[Theorem 2.1]{choi21}. We only provide here a brief outline of the proof as follows. We fix any $s>0$. Recall from~\eqref{c1} and~\eqref{c2} that
$$\gamma_3:=1-2b(a-1)>0,\; \gamma_2:=(1-h-2ab)(a-1)>0,\; \gamma_1:=(1-k-2ab)(a-1)>0.$$

First, let $\uphi_3$ be a positive $C^1(\R)$ solution of
$$\begin{cases}
-s\uphi_3'(z)=d_3\bN_3[\uphi_3](z)+r_3\uphi_3(z)\,[\alpha(z)-2b(a-1)-\uphi_3(z)],\,z\in\bR,\\
\uphi_3(-\infty)=0,\; \uphi_3(+\infty)=\gamma_3.
\end{cases}$$
Such $\uphi_3$ exists, by Proposition~\ref{prop:key} applied with the continuous function $\alpha-2b(a-1)$ instead of~$\alpha$, with limit $\gamma_3=\alpha(+\infty)-2b(a-1)=1-2b(a-1)>0$ at $+\infty$ and limit $\alpha(-\infty)-2b(a-1)<0$ at $-\infty$. Furthermore, by Proposition~\ref{prop:key}, $\uphi_3$ is constructed such that $\uphi_3\le\uphi_3(+\infty)=\gamma_3$ in $\R$, and $\uphi_3$ converges to $\gamma_3$ exponentially at $+\infty$.

Next, note that the continuous bounded function $a\uphi_3-1-h(a-1)$ is such that
$$a\uphi_3(-\infty)-1-h(a-1)=-1-h(a-1)<0,\ \ a\uphi_3(+\infty)-1-h(a-1)=\gamma_2>0,$$
and $a\uphi_3-1-h(a-1)\le a\uphi_3(+\infty)-1-h(a-1)$ in $\R$. By Proposition~\ref{prop:key} again, there exists a positive $C^1(\R)$ solution $\uphi_2$ of
$$\begin{cases}
-s\uphi_2'(z)=d_2\bN_2[\uphi_2](z)+r_2\uphi_2(z)\,[a\uphi_3(z)-1-h(a-1)-\uphi_2(z)],\,z\in\bR,\\
\uphi_2(-\infty)=0,\; \uphi_2(+\infty)=\gamma_2,
\end{cases}$$
such that $\uphi_2\le\uphi_2(+\infty)=\gamma_2$ in $\R$ and $\uphi_2$ converges to $\gamma_2$ exponentially at $+\infty$. Similarly, there exists a positive $C^1(\R)$ solution $\uphi_1$ of
$$\begin{cases}
-s\uphi_1'(z)=d_1\bN_1[\uphi_1](z)+r_1\uphi_1(z)[a\uphi_3(z)-1-k(a-1)-\uphi_1(z)],\,z\in\bR,\\
\uphi_1(-\infty)=0,\; \uphi_1(+\infty)=\gamma_1>0,
\end{cases}$$
such that $\uphi_1\le\uphi_1(+\infty)=\gamma_1$ in $\R$ and $\uphi_1$ converges to $\gamma_1$ exponentially at $+\infty$.

It is straightforward to verify that $(\ophi_1,\ophi_2,\ophi_3)\equiv(a-1,a-1,1)$ and $(\uphi_1,\uphi_2,\uphi_3)$ are a pair of $C^1(\R)$ nonnegative bounded upper-lower-solutions of \eqref{TWS2}, such that $\underline{\phi}_i\le\overline{\phi}_i$ in $\R$ for all $i=1,2,3$ (the last property holds since $(\gamma_1,\gamma_2,\gamma_3)<(a-1,a-1,1)$ componentwise). Hence the existence of positive wave profile $\phitri$ of~\eqref{TWS2} follows from Lemma~\ref{luslem}, with
$$0<\underline{\phi}_i\le\phi_i\le\overline{\phi}_i\ \hbox{ in $\R$ for $i=1,2,3$}.$$

Note that
$$\phi_i^-:=\liminf_{z\to+\infty}\phi_i(z)\ge\gamma_i>0,\; i=1,2,3.$$
With this information, the same proof of \cite[Lemma 3.4]{choi21} can be applied to derive that $\phitri(+\infty)=E_4$. We also refer the reader to the proof of \cite[(4,1)]{gnow20}. Indeed, setting
$$\phi_i^+:=\limsup_{z\to+\infty}\phi_i(z),\; i=1,2,3,$$
picking any $\varepsilon>0$ small enough and considering the family of contracting parallelepipeds $\prod_{i=1}^3[m_i(\theta),M_i(\theta)]$ (with $\theta\in[0,1]$) defined by
$$\left\{\baa{ll}
m_1(\theta):=\theta u^*+(1-\theta)(\gamma_1-\ep), & M_1(\theta):=\theta u^*+(1-\theta)(a-1+\varepsilon),\vspace{3pt}\\
m_2(\theta):=\theta v^*+(1-\theta)(\gamma_2-\ep), & M_2(\theta):=\theta v^*+(1-\theta)(a-1+\varepsilon),\vspace{3pt}\\
m_3(\theta):=\theta w^*+(1-\theta)(\gamma_3-\ep^2), & M_3(\theta):=\theta u^*+(1-\theta)(1+\varepsilon^2),\eaa\right.$$
we can show that $0\in\Theta$ and $\sup\Theta=1$, where
$$\Theta:=\{\theta\in[0,1)\mid m_i(\theta)<\phi_i^-\le\phi_i^+<M_i(\theta),\, i=1,2,3\}.$$
The details can be found in \cite{choi21,gnow20}. We omit them here. Hence $\phitri(+\infty)=(u^*,v^*,w^*)=E_4$, and the proof of Theorem~\ref{th:e4} is complete, thanks to Proposition~\ref{prop:plus}.


\section{Proof of Theorem~\ref{th:e1}}
\setcounter{equation}{0}

Throughout this section, we suppose~\eqref{c1} and~\eqref{alpha1}. For the sufficiency part, let us assume without loss of generality that $s^*_1\ge s^*_2$, with $s^*_1$ and $s^*_2$ defined by~\eqref{defs*i}, and let us then fix $s>s^*_1$. First, we set
\be\label{defgi}
g_i(\ld):=d_i[I_i(\ld)-1]-s\ld+r_i(a-1),\; i=1,2,\;\lambda\in[0,\hat\lambda_i).
\ee
From~\eqref{c1} and (J1)-(J3) (and the comments after~(J1)-(J3)), there exist $0<\ld_1<\ld_2<\hat{\ld}_1$ such that $g_1(\ld_1)=g_1(\ld_2)=0$,  $g_1(\ld)<0$ if and only if $\ld\in(\ld_1,\ld_2)$, and $g_1(\ld)\le0$ if and only if $\ld\in[\ld_1,\ld_2]$. Similarly, since $s>s_2^*$, the equation $g_2(\ld)=0$ has two positive roots $\ld_3,\ld_4$ with $0<\ld_3<\ld_4<\hat{\ld}_2$. Moreover, $g_2(\ld)<0$ if and only if $\ld\in(\ld_3,\ld_4)$, and $g_2(\ld)\le0$ if and only if $\ld\in[\ld_3,\ld_4]$. Set now
\be\label{defg3}
g_3(\ld):=d_3[I_3(\ld)-1]-s\ld.
\ee
Note that $g_3(0)=0$ and $g_3'(0)=-s<0$. Hence we can choose
\be\label{deflambda0}
0<\ld_0<\min\{\ld_1,\ld_3,\rho,\hat\lambda_3\}
\ee
such that
$$g_3(\ld_0)<0,$$
where $\rho>0$ is as in~($\alpha 2$). We then define, for $z\in\R$,
\be\label{deful}\left\{\baa{ll}
\ophi_1(z)=\min\{(a-1)e^{-\ld_1 z}, a-1\}, & \uphi_1(z)=\max\{(a-1)e^{-\ld_1 z}-p_1e^{-\mu_1 z}, 0\},\vspace{3pt}\\
\ophi_2(z)=\min\{(a-1)e^{-\ld_3 z}, a-1\}, & \uphi_2(z)=\max\{(a-1)e^{-\ld_3 z}-p_2e^{-\mu_2 z}, 0\},\vspace{3pt}\\
\ophi_3(z)=1, & \uphi_3(z)=\max\{1-e^{-\ld_0 z}, 0\},\eaa\right.
\ee
where the positive constants $\mu_i,p_i$, $i=1,2$, are chosen (in the following order) such that
\bea
&&\ld_1<\mu_1<\min\{\ld_2,\ld_1+\ld_0\},\label{mu1}\\
&&p_1>\max\left\{a-1, \frac{r_1(a-1)[2a-1+k(a-1)]}{-g_1(\mu_1)}\right\},\label{q1}\\
&&\ld_3<\mu_2<\min\{\ld_4,\ld_3+\lambda_0\},\label{mu2}\\
&&p_2>\max\left\{a-1, \frac{r_2(a-1)[2a-1+h(a-1)]}{-g_2(\mu_2)}\right\}.\label{q2}
\eea
Note that the continuous functions $\underline{\phi}_i$ and $\overline{\phi}_i$ are nonnegative, bounded in $\R$, of class $C^1$ except at finitely many real numbers, and that $\underline{\phi}_i\le\overline{\phi}_i$ in $\R$, for $i=1,2,3$.

Then we have

\begin{lemma}\label{la:1p}
With $s>s^*_1\ge s^*_2$ and the above notations, there exists a positive solution $\phitri$ of \eqref{TWS2} such that $\uphi_i\le\phi_i\le\ophi_i$ in $\bR$, $i=1,2,3$.
\end{lemma}

\begin{proof}
We choose $\ep>0$ small enough such that
\be\label{ep1}
g_3(\ld_0)+\ep r_3<0.
\ee
Also, up to a translation, condition \eqref{a2} holds for this $\ep$.

We now verify \eqref{u1}-\eqref{l3} hold for $(\overline{\phi}_1,\overline{\phi}_2,\overline{\phi}_3)$ and $(\underline{\phi}_1,\underline{\phi}_2,\underline{\phi}_3)$ defined in \eqref{deful} for all $z\in\bR$ except finitely many points. For~\eqref{u1}, the inequality is obvious for $z<0$ since $\ophi_1\le a-1$ in $\R$, $\ophi_1=a-1$ in $(-\infty,0]$, and $\ophi_3\equiv 1$ in $\R$. It then suffices to consider $z>0$. As in the proof of Proposition~\ref{prop:key}, using $\ophi_1(y)\le (a-1)e^{-\ld_1 y}$ for all $y\in\bR$ and $\ophi_1(z)=(a-1)e^{-\lambda_1z}$ for all $z>0$, we have
\be\label{N1ophi11}
d_1\bN_1[\ophi_1](z)+s\ophi_1'(z)\le (a-1)e^{-\ld_1 z}\{d_1[I_1(\ld_1)-1]-s\ld_1\}\hbox{ for all }z>0,
\ee
and thus
\be\label{N1ophi12}\baa{ll}
&d_1\bN_1[\ophi_1](z)+s\ophi_1'(z)+r_{1}\overline{\phi}_1(z)\,[-1-\overline{\phi}_1(z)-k\underline{\phi}_2(z)+a\overline{\phi}_3(z)]\vspace{3pt}\\
\le& (a-1)e^{-\ld_1 z}\{d_1[I_1(\ld_1)-1]-s\ld_1+r_1(-1+a)\}=0.\eaa
\ee
Hence \eqref{u1} holds for all $z\neq 0$. Similarly, \eqref{u2} holds for all $z\neq 0$, using
\be\label{N1ophi13}
g_2(\lambda_3)=d_2[I_2(\ld_3)-1]-s\ld_3+r_2(a-1)=0.
\ee
It is trivial that \eqref{u3} holds for all $z\in\bR$, since $\alpha(z)\le\alpha(+\infty)=1$ and $\uphi_1(z)\ge0$, $\uphi_2(z)\ge0$ for all $z\in\R$.

For \eqref{l1}, since $p_1\!>\!a\!-\!1\!>\!0$ and $\mu_1\!>\!\ld_1\!>\!0$ there is a unique $z_1\!:=\!(\mu_1\!-\!\lambda_1)^{-1}\ln(p_1/(a\!-\!1))\!>\!0$ such that $\uphi_1(z)=0$ for all $z\le z_1$ and
$$\uphi_1(z)=(a-1)e^{-\ld_1 z}-p_1e^{-\mu_1 z}>0\ \hbox{ for all }z>z_1.$$
It is trivial that~\eqref{l1} holds for $z<z_1$. On the other hand, since $\uphi_1(y)\ge(a-1)e^{-\lambda_1y}-p_1e^{-\mu_1y}$ for all $y\in\R$ with equality in $[z_1,+\infty)$, with $0<\lambda_1<\mu_1<\lambda_2<\hat\lambda_1$, it follows that, for every $z>z_1$,
\be\label{N1uphi1}
d_1\bN_1[\uphi_1](z)+s\uphi_1'(z)\ge(a-1)e^{-\ld_1 z}\{d_1[I_1(\ld_1)-1]-s\ld_1\}-p_1e^{-\mu_1 z}\{d_1[I_1(\mu_1)-1]-s\mu_1\},
\ee
and
\beaa
&\!\!&r_{1}\underline{\phi}_1(z)\,[-1-\underline{\phi}_1(z)-k\overline{\phi}_2(z)+a\underline{\phi}_3(z)]\\
&\!\!\ge\!\!&r_1[(a-1)e^{-\ld_1 z}-p_1e^{-\mu_1 z}]\{-1-(a-1)e^{-\ld_1 z}-k(a-1)e^{-\ld_3 z}+a-ae^{-\ld_0 z}\}\\
&\!\!\ge\!\!&r_1(a-1)[(a-1)e^{-\ld_1 z}-p_1e^{-\mu_1 z}]-r_1(a-1)e^{-\ld_1 z}[(a-1)(e^{-\ld_1 z}+ke^{-\lambda_3z})+ae^{-\ld_0 z}].
\eeaa
Hence, since $g_1(\lambda_1)=0$, we have, for $z>z_1$,
\beaa
\mL_1(z)&\!\!\!\ge\!\!\!& e^{-\mu_1 z}\{-p_1g_1(\mu_1)\!-\!r_1(a\!-\!1)[(a\!-\!1)(e^{(\mu_1\!-\!2\ld_1)z}\!+\!ke^{(\mu_1\!-\!\lambda_1\!-\!\lambda_3)z})\!+\!ae^{(\mu_1\!-\!\ld_1\!-\!\ld_0)z}]\}\\
&\!\!\!\ge\!\!\!&e^{-\mu_1 z}\{-p_1g_1(\mu_1)-r_1(a-1)[2a-1+k(a-1)]\}>0,
\eeaa
due to~\eqref{deflambda0} and~\eqref{mu1}-\eqref{q1}. Hence \eqref{l1} holds for all $z\neq z_1$.

Similarly, since $p_2\!>\!a\!-\!1\!>\!0$ and $\mu_2\!>\!\ld_3\!>\!0$, there is a unique $z_2\!:=\!(\mu_2\!-\!\lambda_3)^{-1}\!\ln(p_2/(a\!-\!1))\!>\!0$ such that $\uphi_2(z)=0$ for all $z\le z_2$ and
$$\uphi_2(z)=(a-1)e^{-\ld_3 z}-p_2e^{-\mu_2 z}>0\ \hbox{ for all }z>z_2.$$
Then, as in the previous paragraph, we can easily check that \eqref{l2} holds for all $z\neq z_2$, using~\eqref{deflambda0} and~\eqref{mu2}-\eqref{q2}, together with $g_2(\lambda_3)=0$. We omit the details.

Finally, \eqref{l3} holds trivially for $z<0$. For $z>0$, we compute, using~\eqref{alpha1} and~\eqref{a2},
\beaa
\mL_3(z)&=&d_{3} \bN_3[\underline{\phi}_3](z)+s\underline{\phi}_3^{\prime}(z)+r_{3}\underline{\phi}_3(z)[\alpha(z)-b\overline{\phi}_1(z)-b\overline{\phi}_2(z)-\underline{\phi}_3(z)]\\
&\ge& -e^{-\ld_0 z}\{d_3[I_3(\ld_0)-1]-s\ld_0\}\\
&&+r_3(1-e^{-\ld_0 z})[1-\ep e^{-\rho z}-b(a-1)e^{-\ld_1 z}-b(a-1)e^{-\ld_3 z}-1+e^{-\ld_0 z}]\\
&\ge& -e^{-\ld_0 z}\{d_3[I_3(\ld_0)-1]-s\ld_0\}-\varepsilon r_3(1-e^{-\ld_0 z})e^{-\rho z},
\eeaa
using $2b(a-1)<1$ (by \eqref{c1}) and $e^{-\ld_0 z}\ge e^{-\ld_j z}$ for $z>0$ (since $\ld_0<\ld_j$) for $j=1,3$. Moreover, since $0<\lambda_0<\rho$ by~\eqref{deflambda0}, we deduce that
$$\mL_3(z)\ge -e^{-\ld_0 z}\{d_3[I_3(\ld_0)-1]-s\ld_0+r_3\ep\}>0\ \hbox{ for all }z>0,$$
by the choice of $\ep$ in \eqref{ep1}. Hence \eqref{l3} holds for all $z\neq 0$.

Lemma~\ref{luslem} then yields the existence of a nonnegative bounded solution $(\phi_1,\phi_2,\phi_3)$ of~\eqref{TWS2} such that $\uphi_i\le\phi_i\le\ophi_i$ in $\R$ for $i=1,2,3$. In particular, each function $\phi_i$ is nonnegative and nontrivial, hence it is positive in $\R$ by the strong maximum principle, as explained after Proposition~\ref{prop:plus}. The proof of Lemma~\ref{la:1p} is thereby complete.
\end{proof}

\begin{proof}[Proof of Theorem~$\ref{th:e1}$]
By Lemma~\ref{la:1p} and the definition of the upper-lower-solutions in~\eqref{deful}, the positive solution $(\phi_1,\phi_2,\phi_3)$ of~\eqref{TWS2} given in Lemma~\ref{la:1p} satisfies
$$\phitri(+\infty)=(0,0,1)=E_1.$$
Then, together with Proposition~\ref{prop:plus}, the existence part of Theorem~\ref{th:e1} follows.

For the only if part, we follow the idea of \cite[Theorem 3.5]{choi21}. Assume that there exists a positive solution $\phitri$ of \eqref{TWS2} and \eqref{bc-1} for some $s>0$, and assume that $(\tilde\lambda_i,\hat\lambda_i)=(-\infty,+\infty)$ for some $i\in\{1,2\}$. Without loss of generality, let us assume that $i=1$.
Set $\zeta(z):=\phi_1'(z)/\phi_1(z)$. Then $\zeta$ satisfies
$$-s\zeta(z)=d_1\left[\int_\bR J_1(y)e^{\int_z^{z-y}\zeta(x)dx}dy-1\right]+r_1[-1-\phi_1(z)-k\phi_2(z)+a\phi_3(z)],\; z\in\bR.$$
It follows from~\cite[Proposition 3.7]{zlw12}\footnote{Notice that, in addition to the assumption $(\tilde\lambda_1,\hat\lambda_1)=(-\infty,+\infty)$, $J_1$ is assumed to be even and of class~$C^1(\R)$ in~\cite{zlw12}, but the proof of~\cite[Proposition 3.7]{zlw12} still works under assumptions (J1)-(J3) for $J_1$, under the assumption $(\tilde\lambda_1,\hat\lambda_1)=(-\infty,+\infty)$.} that the limit $-\ld:=\lim_{z\to+\infty}\zeta(z)$ exists in $\R$ (necessarily, $\lambda\ge0$, since $\phi_1>0$ in $\R$ and $\phi_1(+\infty)=0$) and $\ld$ satisfies
$$s\ld=d_1[I_1(\ld)-1]+r_1(a-1).$$
Since $I_1(0)=1$ and $r_1(a-1)>0$, it follows that $\lambda>0$, hence we obtain that $s\ge s_1^*$, owing to the definition of $s^*_1$ in~\eqref{defs*i}. Thereby the proof is complete.
\end{proof}


\section{Proof of Theorem~\ref{th:e2}}
\setcounter{equation}{0}

Throughout this section, we suppose~\eqref{c1} and~\eqref{alpha1}, together with~\eqref{de2}-\eqref{re2}. For the sufficiency part, let us assume that $s\ge R_2(\rho)$ and $s>s_2^{**}$, where $\rho>0$ is the constant in~($\alpha$2), and $R_2$ and $s^{**}_2$ are defined in~\eqref{defs**2}. We recall that $\beta_2>0$ is defined in~\eqref{defbeta2}. From (J1)-(J3) and~\eqref{c1}, the function $G_2$ defined by
\be\label{defG2}
G_2(\ld):=d_2[I_2(\ld)-1]-s\ld+r_2\beta_2,\ \ \lambda\in[0,\hat\lambda_2),
\ee
has two positive zeroes $\ld_5, \ld_6$ such that $0<\ld_5<\ld_6<\hat\lambda_2$. Note that $G_2(\ld)<0$ if and only if $\ld\in(\ld_5,\ld_6)$, and $G_2(\ld)\le0$ if and only if $\ld\in[\ld_5,\ld_6]$. Since $s\ge R_2(\rho)$ and $\rho>0$, one necessarily has $\rho<\hat\lambda_2$ and $G_2(\rho)\le0$, hence $\rho\ge\ld_5$.

We recall that $u_p=(a-1)/(ab+1)$ and $w_p=(b+1)/(ab+1)$ and we define, for $z\in\R$,
\be\label{ou2}\left\{\baa{ll}
\ophi_1(z):=\min\{u_p+Ae^{-\ld_5 z},a-1\}, & \uphi_1(z):=\max\{u_p(1-e^{-\ld_5 z}),0\},\vspace{3pt}\\
\ophi_2(z):=\min\{(a-1)e^{-\ld_5 z},a-1\}, & \uphi_2(z):=\max\{(a-1)e^{-\ld_5 z}-qe^{-\nu z},0\},\vspace{3pt}\\
\ophi_3(z):=\min\{w_p+bu_pe^{-\ld_5 z},1\}, & \uphi_3(z):=\max\{w_p(1-e^{-\ld_5 z}),0\},\eaa\right.
\ee
where
\be\label{defA}
A:=a-1-u_p
\ee
and the positive constants $\nu$ and $q$ are chosen to satisfy
\be
\ld_5<\nu<\min\{\ld_6,2\ld_5\},\ \ q>\max\left\{a-1, \frac{r_2(a-1)[(1+h)(a-1)+aw_p]}{-G_2(\nu)}\right\}.\label{qq}
\ee
Note that $0<u_p<a-1$ (hence, $A>0$), $0<w_p<1$, $w_p+bu_p=1$, and that all functions $\uphi_i$ and $\ophi_i$ are continuous, nonnegative, bounded in $\R$, of class $C^1$ except at finitely many points, with $\underline{\phi}_i\le\overline{\phi}_i$ in $\R$ for $i=1,2,3$.

Then we have

\begin{lemma}\label{la:e2}
With $s>s^{**}_2$, $s\ge R_2(\rho)$ and the above notations, there exists a positive solution $\phitri$ of \eqref{TWS2} such that $\uphi_i\le\phi_i\le\ophi_i$ in $\bR$, $i=1,2,3$.
\end{lemma}

\begin{proof}
Choose $\ep\in(0,r_2\beta_2/r_3)$ so that, up to a translation, \eqref{a2} holds for this $\ep$. We verify \eqref{u1}-\eqref{l3} hold for $(\overline{\phi}_1,\overline{\phi}_2,\overline{\phi}_3)$ and $(\underline{\phi}_1,\underline{\phi}_2,\underline{\phi}_3)$ defined in \eqref{ou2} for all $z\in\bR$ except finitely many points.

That \eqref{u1}-\eqref{u3} hold for $z<0$ is trivial. For $z>0$, since
$$-1-\overline{\phi}_1(z)-k\underline{\phi}_2(z)+a\overline{\phi}_3(z)\le -1-u_p-Ae^{-\ld_5 z}+aw_p+abu_pe^{-\ld_5 z}=0,$$
we obtain
$$\mU_1(z)\le Ae^{-\ld_5 z}\{d_1[I_1(\ld_5)\!-\!1]\!-\!s\ld_5\}\le Ae^{-\ld_5 z}\{d_2[I_2(\ld_5)\!-\!1]\!-\!s\ld_5\}<Ae^{-\lambda_5z}G_2(\lambda_5)=0,$$
since $I_1(\lambda_5)-1\ge0$ (from the general properties explained after~(J1)-(J3)) and since $d_1\le d_2$ and $J_1=J_2$ in $\bR$ by~\eqref{de2}. Hence \eqref{u1} holds for all $z\neq 0$.

For $z>0$, we compute
\beaa
\mU_2(z)&\le&(a-1)e^{-\ld_5 z}\{d_2[I_2(\ld_5)-1]-s\ld_5\}\\
&& +r_2(a-1)e^{-\ld_5 z}\{(-1-hu_p+aw_p)+[hu_p-(a-1)+abu_p]e^{-\ld_5 z}\}\\
&=&r_2(a-1)e^{-2\ld_5 z}[hu_p-(a-1)+abu_p]\le 0,
\eeaa
using $-1-hu_p+aw_p=\beta_2$, $G_2(\ld_5)=0$ and $hu_p-(a-1)+abu_p<(1+ab)u_p-(a-1)=0$. Hence \eqref{u2} holds for all $z\neq 0$.

Still for $z>0$, we have
$$\alpha(z)-b\uphi_1(z)-b\uphi_2(z)-\ophi_3(z)\le 1-bu_p(1-e^{-\ld_5 z})-w_p-bu_pe^{-\ld_5 z}=0,$$
hence
$$\mU_3(z)\le bu_pe^{-\ld_5 z}\{d_3[I_3(\ld_5)-1]-s\ld_5\}\le bu_pe^{-\ld_5 z}\{d_2[I_2(\ld_5)-1]-s\ld_5\}<0,$$
since $I_3(\lambda_5)-1\ge0$, $d_3\le d_2$ and $J_3=J_2$ in $\bR$, by~\eqref{de2}. Hence \eqref{u3} holds for all $z\neq 0$.

For~\eqref{l1}, it is trivial for $z<0$. For $z>0$, we have
$$d_1\bN_1[\uphi_1](z)+s\uphi_1'(z)\ge -u_pe^{-\ld_5 z}\{d_1[I_1(\ld_5)-1]-s\ld_5\}$$
and
$$r_1\uphi_1(z)\,[-1-\uphi_1(z)-k\ophi_2(z)+a\uphi_3(z)]=r_1u_p(1-e^{-\ld_5 z})[-1-k(a-1)]e^{-\ld_5 z},$$
using $-1-u_p+aw_p=0$. Hence we obtain
\beaa
\mL_1(z)&\ge& -u_pe^{-\ld_5 z}\{d_1[I_1(\ld_5)-1]-s\ld_5+r_1[1+k(a-1)]\}\\
&\ge& -u_pe^{-\ld_5 z}\{d_2[I_2(\ld_5)-1]-s\ld_5+r_2\beta_2]\}=0,
\eeaa
since $I_1(\lambda_5)-1\ge0$, $d_1\le d_2$, $J_1=J_2$ in $\bR$ and $r_1[1+k(a-1)]\le r_2\beta_2$, by~\eqref{de2}-\eqref{re2}. Hence~\eqref{l1} holds for all $z\neq 0$.

Note that, since $q>a-1$ and $\nu>\lambda_5$, there is a unique $z_1:=(\nu-\lambda_5)^{-1}\ln(q/(a-1))>0$ such that
\beaa
\uphi_2(z)=
\begin{cases}
\ (a-1)e^{-\ld_5 z}-qe^{-\nu z}>0, & z>z_1,\\
\ 0, & z\le z_1.
\end{cases}
\eeaa
Clearly, \eqref{l2} holds for $z<z_1$. For $z>z_1$, we compute
$$d_2\bN_2[\uphi_2](z)+s\uphi_2'(z)\ge (a-1)e^{-\ld_5 z}\{d_2[I_2(\ld_5)-1]-s\ld_5\}-qe^{-\nu z}\{d_2[I_2(\nu)-1]-s\nu\},$$
using $\uphi_2(y)\ge (a-1)e^{-\ld_5 y}-qe^{-\nu y}$ for all $y\in\bR$, with equality at $z$. Also, we have, since $\beta_2=-1-hu_p+aw_p$,
\beaa
&&r_{2}\underline{\phi}_2(z)[-1-h\overline{\phi}_1(z)-\underline{\phi}_2(z)+a\underline{\phi}_3(z)]\\
&=&r_2\uphi_2(z)\{\beta_2-h(a-1)e^{-\ld_5 z}+hu_pe^{-\ld_5 z}-(a-1)e^{-\ld_5 z}+qe^{-\nu z}-aw_pe^{-\ld_5 z}\}\\
&\ge&r_2[(a-1)e^{-\ld_5 z}-qe^{-\nu z}]\{\beta_2-[(1+h)(a-1)+aw_p]e^{-\ld_5 z}\}\\
&\ge&r_2\beta_2[(a-1)e^{-\ld_5 z}-qe^{-\nu z}]-r_2(a-1)[(1+h)(a-1)+aw_p]e^{-2\ld_5 z},
\eeaa
for all $z>z_1$. Hence, using $G_2(\lambda_5)=0$, we obtain, for $z>z_1$,
\beaa
\mL_2(z)&\ge& e^{-\nu z}\{-qG_2(\nu)-r_2(a-1)[(1+h)(a-1)+aw_p]e^{(\nu-2\ld_5) z}\}\\
&\ge&e^{-\nu z}\{-qG_2(\nu)-r_2(a-1)[(1+h)(a-1)+aw_p]\}>0,
\eeaa
using $\nu<2\ld_5$ and \eqref{qq}. Hence \eqref{l2} holds for all $z\neq z_1$.

Finally, we consider \eqref{l3}. It suffices to consider the case $z>0$. For $z>0$, we have $\uphi_3(z)=w_p-w_pe^{-\ld_5 z}>0$, while $\uphi_3(y)\ge w_p-w_pe^{-\ld_5 y}$ for all $y\in\bR$. Hence
\beaa
d_3\bN_3[\uphi_3](z)+s\uphi_3'(z)&\ge& -w_pe^{-\ld_5 z}\{d_3[I_3(\ld_5)-1]-s\ld_5\}\\
&\ge& -w_p e^{-\ld_5 z}\{d_2[I_2(\ld_5)-1]-s\ld_5\}=r_2\beta_2w_p e^{-\ld_5 z},
\eeaa
since $I_3(\lambda_5)-1\ge0$, $d_3\le d_2$, $J_3=J_2$ by~\eqref{de2}, and $G_2(\ld_5)=0$. Also, for $z>0$, using $1-bu_p-w_p=0$, $A=a-1-u_p$ and $2b(a-1)<1$, together with~\eqref{alpha1} and~\eqref{a2}, we compute
\beaa
r_{3}\underline{\phi}_3(z)[\alpha(z)-b\overline{\phi}_1(z)-b\overline{\phi}_2(z)-\underline{\phi}_3(z)] &\ge& r_3\uphi_3(z)\{-\ep e^{-\rho z}+[1-2b(a-1)]e^{-\ld_5 z}\}\\
&\ge& -r_3w_p\ep e^{-\rho z}\ge -r_3w_p\ep e^{-\ld_5 z},
\eeaa
due to $\rho\ge\ld_5$. Hence, for $z>0$, $\mL_3(z)\ge r_2\beta_2w_p e^{-\ld_5 z}-r_3w_p\ep e^{-\ld_5 z}>0$ since $\ep<r_2\beta_2/r_3$. We conclude that \eqref{l3} holds for all $z\neq 0$.

Since $0\le\uphi_i\le\ophi_i$ in~$\bR$ for $i=1,2,3$, and each function $\uphi_i$ is nontrivial, the conclusion of Lemma~\ref{la:e2} follows from Lemma~\ref{luslem}, as at the end of the proof of Lemma~\ref{la:1p}.
\end{proof}

With Lemma~\ref{la:e2}, Theorem~\ref{th:e2} can be proved in the same way as that of Theorem~\ref{th:e1}. We safely omit it here.


\section{Waves with critical speed}
\setcounter{equation}{0}

This section is devoted to the proofs of Theorems~\ref{th:c1}-\ref{th:c2} on the existence of forced waves with critical speeds. Three cases are considered: two of them are concerned with waves connecting the trivial state $(0,0,0)$ and the predator-free state $E_1=(0,0,1)$, and the last one is concerned with waves connecting the trivial state $(0,0,0)$ and the mixed state $E_2=(u_p,0,w_p)$. Throughout this section, in addition to~\eqref{c1} and~\eqref{alpha1}, we assume that $J_1$ is compactly supported. Hence there is a positive constant $S>0$ such that
\be\label{SS}
J_1(y)=0\ \hbox{ for almost every } |y|>S.
\ee
In particular, $\tilde{\lambda}_1=-\infty$ and $\hat{\lambda}_1=+\infty$ in~(J3).


\subsection{Waves connecting $(0,0,0)$ and $E_1=(0,0,1)$ in the case $s^*_1=s^*_2$: proof of Theorem~\ref{th:c1}}

In this subsection, we assume that $s^*_1=s^*_2$, with $s^*_i>0$ defined in~\eqref{defs*i}, and that $J_2$ is also compactly supported. Even if it means increasing $S>0$, we can assume without loss of generality that
\be\label{SS2}
J_2(y)=0\ \hbox{ for almost every } |y|>S.
\ee
We consider the critical speed
$$s=s^*_1=s^*_2.$$
For $i=1,2$, let $\lambda^*_i>0$ be the unique minimum of $Q_i$ in $(0,\hat\lambda_i)=(0,+\infty)$ given by~\eqref{defs*i}. With $g_i$ defined in~\eqref{defgi} with $s=s^*_1=s^*_2$, one has $g_i(\lambda^*_i)=0$, and $\lambda^*_i$ is the unique positive root of this equation. Since $Q'_i(\lambda^*_i)=0$ for $i=1,2$, we also have
\be\label{gi-d}
s=s^*_1=s^*_2=d_1\int_{\bR}J_1(y)ye^{\lambda^*_1y}dy=d_2\int_{\bR}J_2(y)ye^{\lambda^*_2y}dy.
\ee
With $\rho>0$ as in~($\alpha 2$) and $g_3$ as in~\eqref{defg3}, we choose a constant $\ld_0>0$ such that
$$0<\lambda_0<\min\{\lambda^*_1,\lambda^*_2,\rho\}\ \hbox{ and }\ g_3(\ld_0)=d_3[I_3(\ld_0)-1]-s\ld_0<0.$$

We then define
\be\label{cu21}\left\{\baa{l}
\ophi_1(z)=\begin{cases}(a-1)Bze^{-\ld^*_1z}, & z>z_1,\\ a-1, & z\le z_1,\end{cases}\ \ \ \ \ \ \ophi_2(z)=\begin{cases}(a-1)Bze^{-\ld^*_2z}, & z>z_2,\\ a-1, & z\le z_2,\end{cases}\vspace{5pt}\\
\uphi_1(z)=\begin{cases}(a-1)Bze^{-\lambda^*_1z}-p_3z^{1/2}e^{-\lambda^*_1z}, & z>z_3,\\ 0, & z\le z_3,\end{cases}\vspace{5pt}\\
\uphi_2(z)=\begin{cases}(a-1)Bze^{-\lambda^*_2z}-p_4z^{1/2}e^{-\lambda^*_2z}, & z>z_4,\\ 0, & z\le z_4,\end{cases}\vspace{5pt}\\
\ophi_3(z)=1,\ \ \uphi_3(z)=\max\{1-e^{-\lambda_0(z-z_0)},0\},\eaa\right.
\ee
where the parameters $B,z_1,z_2,z_0,p_3,z_3,p_4,z_4$ are chosen in the following order:
\begin{itemize}
\item $B>\max\{\lambda^*_1e,\lambda^*_2e\}$ large enough such that $z_1-z'_1>S$ and $z_2-z'_2>S$, where $0<z'_1<1/\lambda^*_1<z_1$,  $0<z'_2<1/\lambda^*_2<z_2$, and
\be\label{defB}
B=\frac{e^{\lambda^*_1z_1}}{z_1}=\frac{e^{\lambda^*_1z'_1}}{z'_1}=\frac{e^{\lambda^*_2z_2}}{z_2}=\frac{e^{\lambda^*_2z'_2}}{z'_2}
\ee
(notice that the smallest positive root $z'_i$ of each equation $B=e^{\lambda^*_iz}/z$ actually does not appear in~\eqref{cu21}, but it will play a role in the proof of Lemma~\ref{la:c1} below; observe also that $e^{\lambda^*_iz}/z\le B$ for $z\in[z'_1,z_i]$, and that $z_i>z'_i+S>S$, for $i=1,2$);
\item $z_0>\max\{z_1,z_2\}$ such that
\be\label{defz0}
b(a-1)Bz\big(e^{-\lambda^*_1z}+e^{-\lambda^*_2z}\big)\le e^{-\lambda_0(z-z_0)}\ \hbox{ for all $z\ge z_0$}
\ee
(the choice of $z_0$ is possible since $0<\lambda_0<\min\{\lambda^*_1,\lambda^*_2\}$);
\item $p_3>0$ large enough such that $z_3:=\{p_3/[(a-1)B]\}^2>z_0$ and
\be\label{p3}
p_3>\frac{8r_1B(a\!-\!1)\!\times\!\displaystyle\mathop{\max}_{z\ge0}\left\{(z\!+\!S)^{3/2}[(a\!-\!1)Bz^2(e^{-\lambda^*_1z}\!+\!ke^{-\lambda^*_2z})\!+\!aze^{-\ld_0(z\!-\!z_0)}]\right\}}{\displaystyle d_1\int_\bR J_1(y)y^2e^{\lambda^*_1y}dy};
\ee
\item $p_4>0$ large enough such that $z_4:=\{p_4/[(a-1)B]\}^2>z_0$ and
\be\label{p4}
p_4\!>\!\frac{8r_2B(a\!-\!1)\!\times\!\displaystyle\mathop{\max}_{z\ge0}\left\{(z\!+\!S)^{3/2}[(a\!-\!1)Bz^2(e^{-\lambda^*_2z}\!+\!he^{-\lambda^*_1z})\!+\!aze^{-\ld_0(z\!-\!z_0)}]\right\}}{\displaystyle d_2\int_\bR J_2(y)y^2e^{\lambda^*_2y}dy}.
\ee
\end{itemize}
It is straightforward to check that the functions $\uphi_i$ and $\ophi_i$ are continuous in $\R$, nonnegative, bounded, of class $C^1$ except at finitely many points, and that $\uphi_i\le\ophi_i$ in $\R$, for $i=1,2,3$.

\begin{lemma}\label{la:c1}
Under the above assumptions and notations, there exists a positive solution $\phitri$ of~\eqref{TWS2} such that $\uphi_i\le\phi_i\le\ophi_i$ in $\bR$, $i=1,2,3$.
\end{lemma}

\begin{proof}
We choose $\ep>0$ small enough such that
\be\label{ep22}
e^{\lambda_0z_0}g_3(\ld_0)+\ep r_3<0.
\ee
Also, up to a translation, condition \eqref{a2} holds for this $\ep$.

We now verify \eqref{u1}-\eqref{l3} for $(\overline{\phi}_1,\overline{\phi}_2,\overline{\phi}_3)$ and $(\underline{\phi}_1,\underline{\phi}_2,\underline{\phi}_3)$ defined in~\eqref{cu21}, for all $z\in\bR$ except finitely many points. First of all, it is trivial that~\eqref{u1} and~\eqref{u2} hold for $z<z_1$ and $z<z_2$ respectively, since $\ophi_1=a-1$ in $(-\infty,z_1)$, $\ophi_2=a-1$ in $(-\infty,z_2)$, and since $\ophi_1\le a-1$, $\ophi_2\le a-1$, $\ophi_3=1$, $\uphi_1\ge0$ and $\uphi_2\ge0$ in $\R$. For $z>z_1$, we have $\ophi_1(z)=(a-1)Bze^{-\lambda^*_1z}$ and, since $z-y>z_1-S>z'_1$ for all $y\in[-S,S]$, we also have
$$\ophi_1(z-y)\le (a-1)B(z-y)e^{-\lambda^*_1(z-y)}\ \hbox{ for all }y\in[-S,S]$$
(indeed, if $z-y>z_1$, we have equality in the above inequality, and if $z'_1<z-y\le z_1$, we have $e^{\lambda^*_1(z-y)}/(z-y)\le B$ and $\ophi_1(z-y)=a-1\le(a-1)B(z-y)e^{-\lambda^*_1(z-y)}$). Hence, for all $z>z_1$, using $J_1=0$ almost everywhere outside $[-S,S]$,
\be\label{N1ophi1}\baa{rcl}
\bN_1[\ophi_1](z)&\!\!\le\!\!& \displaystyle(a-1)B\left[\int_{-S}^SJ_1(y)(z-y)e^{-\lambda^*_1(z-y)}dy-ze^{-\lambda^*_1z}\right]\vspace{3pt}\\
&\!\!=\!\!&\displaystyle(a-1)B\left[\int_{\bR}J_1(y)(z-y)e^{-\lambda^*_1(z-y)}dy-ze^{-\lambda^*_1z}\right]\vspace{3pt}\\
&\!\!=\!\!&\displaystyle(a\!-\!1)Bze^{-\lambda^*_1z}\!\left[\int_\bR\! J_1(y)e^{\lambda^*_1y}dy-1\right]-(a\!-\!1)Be^{-\lambda^*_1z}\!\!\int_\bR J_1(y)ye^{\lambda^*_1y}dy.\eaa
\ee
This implies, using $\ophi_1(z)\ge 0$, $\uphi_2(z)\ge 0$ and $\ophi_3(z)=1$,
\be\label{U1ophi1}\baa{rcl}
\mU_1(z)&\le& \displaystyle(a-1)Bze^{-\lambda^*_1z}\left\{d_1\left[I_1(\lambda^*_1)-1\right]-s\lambda^*_1+r_1(a-1)\right\}\vspace{3pt}\\
&&\displaystyle-(a-1)Be^{-\lambda^*_1z}\left[d_1\int_\bR J_1(y)ye^{\lambda^*_1y}dy-s\right]=0,\eaa
\ee
because of~\eqref{gi-d} and $g_1(\lambda^*_1)=0$. Hence \eqref{u1} holds for all $z\neq z_1$. Similarly, we have $\mU_2(z)\le0$ for all $z>z_2$, because $g_2(\lambda^*_2)=0$ and~\eqref{SS2}-\eqref{gi-d} hold. Hence \eqref{u2} holds for all $z\neq z_2$. That \eqref{u3} holds for all $z\in\bR$ is trivial, since $\alpha(z)\le 1$, $\ophi_3(z)=1$, $\uphi_1(z)\ge0$ and $\uphi_2(z)\ge0$ for all $z\in\R$.

It remains to check~\eqref{l1}-\eqref{l3}. Since $\uphi_1=0$ in $(-\infty,z_3]$ and $\uphi_1\ge0$ in $\R$,~\eqref{l1} clearly holds for $z<z_3$. For $z>z_3\ (>0)$, we have $\uphi_1(z)=(a-1)Bze^{-\lambda^*_1z}-p_3z^{1/2}e^{-\lambda^*_1z}$. Note also from the definitions of $\uphi_1$ and $z_3$ that
$$\uphi_1(x)\ge (a-1)Bxe^{-\lambda^*_1x}-p_3x^{1/2}e^{-\lambda^*_1x}\ \hbox{ for all }x>0.$$
Then, using $z>z_3>z_0>z_1>S$, we have
$$\uphi_1(z-y)J_1(y)\ge\left[(a\!-\!1)B(z\!-\!y)e^{-\lambda^*_1(z\!-\!y)}\!-\!p_3(z\!-\!y)^{1/2}e^{-\lambda^*_1(z\!-\!y)}\right]\!J_1(y)\hbox{ for all }y\in[-S,S].$$
Using $J_1(y)=0$ for almost every $y\not\in[-S,S]$, it follows that, for all $z>z_3$,
\be\label{N1uphi1bis}\baa{rcl}
d_1\bN_1[\uphi_1](z)\!+\!s\uphi_1'(z)&\!\!\!\ge\!\!\!&(a-1)Bze^{-\lambda^*_1z}\{d_1[I_1(\lambda^*_1)-1]-s\lambda^*_1\}\vspace{3pt}\\
&\!\!\!\!\!\!&\displaystyle-(a-1)Be^{-\lambda^*_1z}\left\{ d_1\!\int_{\bR}\! J_1(y)ye^{\lambda^*_1y}dy-s\right\}+p_3d_1z^{1/2}e^{-\lambda^*_1z}\vspace{3pt}\\
&\!\!\!\!\!\!&\displaystyle-p_3e^{-\lambda^*_1z}\left\{d_1\!\int_\bR\! J_1(y)(z-y)^{1/2} e^{\lambda^*_1y}dy-s\lambda^*_1z^{1/2}+\frac{s}{2z^{1/2}} \right\}.\eaa
\ee
On the other hand, we have, for all $z>z_3\ (>z_0>\max\{z_1,z_2\})$,
\beaa
&&r_{1}\underline{\phi}_1(z)[-1-\underline{\phi}_1(z)-k\overline{\phi}_2(z)+a\underline{\phi}_3(z)]\\
&\ge& r_1(a-1)\uphi_1(z)-r_1\ophi_1^2(z)-r_1k\ophi_1(z)\ophi_2(z)-r_1a\ophi_1(z)e^{-\ld_0(z-z_0)}\\
&=&r_1(a-1)^2Bze^{-\lambda^*_1z}-r_1(a-1)p_3z^{1/2}e^{-\lambda^*_1z}-r_1(a-1)^2B^2z^2(e^{-2\lambda^*_1z}+ke^{-(\lambda^*_1+\lambda^*_2)z})\\
& &-r_1a(a-1)Bze^{-\lambda^*_1z-\ld_0(z-z_0)}.
\eeaa
Hence we deduce from $g_1(\lambda^*_1)=0$ and \eqref{gi-d} that, for all $z>z_3$,
$$\mL_1(z)\ge e^{-\lambda^*_1z}[p_3A_1(z)-A_2(z)],$$
where
\be\label{defA12}\begin{cases}
\displaystyle A_1(z):=d_1\!\left[z^{1/2}-\!\int_\bR\! J_1(y)(z-y)^{1/2}e^{\lambda^*_1y}dy\right]+s\lambda^*_1z^{1/2}-\frac{s}{2z^{1/2}}-r_1(a-1)z^{1/2},\vspace{3pt}\\
A_2(z):=r_1(a-1)^2B^2z^2(e^{-\lambda^*_1z}+ke^{-\lambda^*_2z})+r_1a(a-1)Bze^{-\ld_0(z-z_0)}.\end{cases}
\ee
Now, using $g_1(\lambda^*_1)=0$ and \eqref{gi-d} again, we may rewrite $A_1$ as
$$A_1(z)\!=\!d_1\!\!\int_\bR\!\! J_1(y)\!\!\left[z^{1/2}\!\!-\!(z\!-\!y)^{1/2}\!\!-\!\frac{y}{2z^{1/2}}\right]\!e^{\lambda^*_1y}dy\!=\!d_1\!\!\int_{-S}^S\!\!J_1(y)\!\!\left[z^{1/2}\!\!-\!(z\!-\!y)^{1/2}\!\!-\!\frac{y}{2z^{1/2}}\right]\!e^{\lambda^*_1y}dy.$$
Since
$$z^{1/2}-(z-y)^{1/2}-\frac{y}{2z^{1/2}}=\frac{y^2}{2z^{1/2}(z^{1/2}+(z-y)^{1/2})^2}\ge\frac{y^2}{8(z+S)^{3/2}}\ \mbox{ for $y\in[-S,S]$,}$$
we obtain that
\be\label{ineqA1}
A_1(z)\ge \frac{d_1}{8(z+S)^{3/2}}\int_\bR J_1(y)y^2e^{\lambda^*_1y}dy\ \hbox{ for all }z>z_3.
\ee
Therefore, $\mL_1(z)\ge 0$ for all $z>z_3$, by the choice of $p_3$ in \eqref{p3}. Hence \eqref{l1} holds for all $z\neq z_3$. Similarly, \eqref{l2} holds for all $z\neq z_4$, using especially $g_2(\lambda^*_2)=0$ together with~\eqref{gi-d} and~\eqref{p4}.

Finally, \eqref{l3} holds trivially for $z<z_0$. For $z>z_0$, using $z_0>\max\{z_1,z_2\}>0$ together with~\eqref{alpha1},~\eqref{a2},~\eqref{defz0} and $\ld_0<\rho$, we have
$$\baa{rcl}
\alpha(z)\!-\!b\ophi_1(z)\!-\!b\ophi_2(z)\!-\!\uphi_3(z) & \!\!\!\ge\!\!\! & 1\!-\!\ep e^{-\rho z}\!-\!b(a\!-\!1)Bz(e^{-\lambda^*_1z}\!+\!e^{-\lambda^*_2z})\!-\!1\!+\!e^{-\ld_0(z\!-\!z_0)}\vspace{3pt}\\
& \!\!\!\ge\!\!\! & -\ep e^{-\rho z}\ge-\ep e^{-\lambda_0 z}.\eaa$$
It follows that, for all $z>z_0$,
\be\label{mL3}\baa{rcl}
\mL_3(z)&\ge& -e^{-\ld_0(z-z_0)}\left\{d_3[I_3(\ld_0)-1]-s\ld_0\right\}-r_3\ep[1-e^{-\ld_0(z-z_0)}]e^{-\lambda_0 z}\vspace{3pt}\\
&\ge& -e^{-\ld_0
z}[e^{\lambda_0z_0}g_3(\lambda_0)+r_3\ep]\ge 0,\eaa
\ee
due to \eqref{ep22}. Hence \eqref{l3} holds for all $z\neq z_0$.

Since $0\le\uphi_i\le\ophi_i$ in~$\bR$ for $i=1,2,3$, and each function $\uphi_i$ is nontrivial, the conclusion of Lemma~\ref{la:c1} follows from Lemma~\ref{luslem}, as at the end of the proof of Lemma~\ref{la:1p}.
\end{proof}

Clearly, the solution $\phitri$ given in Lemma~\ref{la:c1} satisfies $\phitri(+\infty)=(0,0,1)$. Together with Proposition~\ref{prop:plus}, Theorem~\ref{th:c1} follows.


\subsection{Waves connecting $(0,0,0)$ and $E_1=(0,0,1)$ in the case $s^*_1>s^*_2$: proof of Theorem~\ref{th:c12}}

In this subsection, in addition to~\eqref{SS}, we assume that $s^*_1>s^*_2$, with $s^*_i>0$ defined in~\eqref{defs*i}, and we consider the critical speed
$$s=s^*_1.$$
We recall that $\lambda^*_1>0$ is the unique minimum of $Q_1$ in $(0,\hat\lambda_1)=(0,+\infty)$. Similarly to~\eqref{gi-d}, we also have
\be\label{gi-d1}
s=s^*_1=d_1\int_{\bR}J_1(y)ye^{\lambda^*_1y}dy.
\ee
Since $s=s^*_1>s^*_2$, with $g_2$ given by~\eqref{defgi} in $[0,\hat\lambda_2)$, it then follows from~\eqref{c1} and (J1)-(J3) (and the comments after~(J1)-(J3)) that there exist $0<\ld_3<\ld_4<\hat\lambda_2$ such that
$$g_2(\ld_3)=g_2(\ld_4)=0,$$
and $g_2(\lambda)<0$ if and only if $\lambda\in(\lambda_3,\lambda_4)$. With $g_3$ given by~\eqref{defg3} in $[0,\hat\lambda_3)$, we have $g_3(0)=0$ and $g_3'(0)=-s<0$. Hence we can choose a constant $\ld_0>0$ such that
\be\label{deflambda0bis}
0<\ld_0<\min\{\lambda^*_1,\ld_3,\rho,\hat\lambda_3\}\ \hbox{ and }\ g_3(\ld_0)<0,
\ee
where $\rho>0$ is as in~($\alpha 2$). We also fix a real number $\mu_3$ such that
$$\lambda_3<\mu_3<\min\{\lambda_4,\lambda_3+\lambda_0\}.$$

We then define
\be\label{cu212}\left\{\baa{l}
\ophi_1(z)=\begin{cases}(a-1)Bze^{-\ld^*_1z}, & z>z_1,\\ a-1, & z\le z_1,\end{cases}\vspace{5pt}\\
\uphi_1(z)=\begin{cases}(a-1)Bze^{-\lambda^*_1z}-p_5z^{1/2}e^{-\lambda^*_1z}, & z>z_5,\\ 0, & z\le z_5,\end{cases}\vspace{5pt}\\
\ophi_2(z)=\min\{(a-1)e^{-\lambda_3z},a-1\},\ \ \ \ \uphi_2(z)=\max\{(a-1)e^{-\lambda_3z}-p_6e^{-\mu_3z},0\},\vspace{5pt}\\
\ophi_3(z)=1,\qquad\qquad\qquad\qquad\qquad\ \ \ \ \,\uphi_3(z)=\max\{1-e^{-\lambda_0(z-z_0)},0\},\eaa\right.
\ee
where the parameters $B,z_1,z_0,p_5,z_5,p_6$ are chosen in the following order:
\begin{itemize}
\item $B>\lambda^*_1e$ large enough such that $z_1-z'_1>S$, where $S$ is as in~\eqref{SS} and $0<z'_1<1/\lambda^*_1<z_1$ are defined by
$$B=\frac{e^{\lambda^*_1z_1}}{z_1}=\frac{e^{\lambda^*_1z'_1}}{z'_1};$$
\item $z_0>z_1$ such that
\be\label{defz0bis}
b(a-1)Bze^{-\lambda^*_1z}+b(a-1)e^{-\lambda_3z}\le e^{-\lambda_0(z-z_0)}\ \hbox{ for all $z\ge z_0$}
\ee
(the choice of $z_0$ is possible since $0<\lambda_0<\min\{\lambda_3,\lambda^*_1\}$);
\item $p_5>0$ large enough such that $z_5:=\{p_5/[(a-1)B]\}^2>z_0$ and
\be\label{p5}
p_5\!>\!\frac{8r_1(a\!-\!1)B\!\times\!\displaystyle\mathop{\max}_{z\ge0}\left\{\!(z\!+\!S)^{3/2}[(a\!-\!1)Bz^2e^{-\lambda^*_1z}\!+\!k(a\!-\!1)ze^{-\lambda_3z}\!+\!aze^{-\ld_0(z\!-\!z_0)}]\right\}}{\displaystyle d_1\int_\bR J_1(y)y^2e^{\lambda^*_1y}dy};
\ee
\item $p_6>0$ large enough such that $z_6:=(\mu_3-\lambda_3)^{-1}\ln(p_6/(a-1))>z_0$ and
\be\label{p6}
p_6>\max\left\{a-1,\frac{r_2(a-1)P_6}{-g_2(\mu_3)}\right\},
\ee
with
$$P_6:=\mathop{\max}_{z\ge0}\left\{h(a\!-\!1)Bze^{-(\lambda^*_1+\lambda_3-\mu_3)z}\!+\!(a\!-\!1)e^{-(2\lambda_3-\mu_3)z}\!+\!ae^{-\lambda_3z-\lambda_0(z-z_0)+\mu_3z}\right\}$$
(the choice of $p_6$ is possible since $\mu_3<\lambda_3+\lambda_0<\min\{2\lambda_3,\lambda^*_1+\lambda_3\}$).
\end{itemize}
It is straightforward to check that the functions $\uphi_i$ and $\ophi_i$ are continuous in $\R$, nonnegative, bounded, of class $C^1$ except at finitely many points, and that $\uphi_i\le\ophi_i$ in $\R$, for $i=1,2,3$.

\begin{lemma}\label{la:c12}
Under the above assumptions and notations, there exists a positive solution $\phitri$ of \eqref{TWS2} such that $\uphi_i\le\phi_i\le\ophi_i$ in $\bR$, $i=1,2,3$.
\end{lemma}

\begin{proof}
As in the beginning of the proof of Lemma~\ref{la:c1}, we first choose $\ep>0$ small enough such that~\eqref{ep22} holds, as well as condition~\eqref{a2}, up to a translation.

We now verify \eqref{u1}-\eqref{l3} for $(\overline{\phi}_1,\overline{\phi}_2,\overline{\phi}_3)$ and $(\underline{\phi}_1,\underline{\phi}_2,\underline{\phi}_3)$ defined in \eqref{cu212}, for all $z\in\bR$ except finitely many points. Firstly, it is trivial that~\eqref{u1} holds for $z<z_1$, since $\ophi_1\le a-1$, $\ophi_3=1$, $\uphi_1\ge0$ and $\uphi_2\ge0$ in $\R$. Secondly, it follows from~\eqref{gi-d1} and $g_1(\lambda^*_1)=0$ that~\eqref{u1} holds for all $z>z_1$, doing as in~\eqref{N1ophi1}-\eqref{U1ophi1} in the proof of Lemma~\ref{la:c1}. Hence \eqref{u1} holds for all $z\neq z_1$. Thirdly, as in~\eqref{N1ophi11}-\eqref{N1ophi13} in the proof of Lemma~\ref{la:1p}, we have $\mU_2(z)\le 0$, i.e.~\eqref{u2}, for all $z\neq 0$, using $g_2(\ld_3)=0$. Fourthly, that \eqref{u3} holds for all $z\in\bR$ is trivial, since $\alpha\le 1$, $\uphi_1\ge0$ and $\uphi_2\ge0$ in $\R$.

It remains to show~\eqref{l1}-\eqref{l3}. Since $\uphi_1(z)=0$ for all $z\le z_5$ and $\uphi_1\ge0$ in $\R$,~\eqref{l1} clearly holds for $z<z_5$. For $z>z_5\ (>z_0>0)$, we have $\uphi_1(z)=(a-1)Bze^{-\lambda^*_1z}-p_5z^{1/2}e^{-\lambda^*_1z}$. Note also from the definitions of $\uphi_1$ and $z_5$ that
$$\uphi_1(x)\ge (a-1)Bxe^{-\lambda^*_1x}-p_5x^{1/2}e^{-\lambda^*_1x}\ \hbox{ for all }x>0.$$
Then, using $z>z_5>z_0>z_1>S$, we have
$$\uphi_1(z-y)J_1(y)\ge\left[(a\!-\!1)B(z\!-\!y)e^{-\lambda^*_1(z\!-\!y)}\!-\!p_5(z\!-\!y)^{1/2}e^{-\lambda^*_1(z\!-\!y)}\right]\!J_1(y)\hbox{ for all }y\in[-S,S].$$
Using $J_1(y)=0$ for almost every $y\not\in[-S,S]$, it follows that, for all $z>z_5$,
\beaa
d_1\bN_1[\uphi_1](z)+s\uphi_1'(z)&\!\!\!\ge\!\!\!&(a-1)Bze^{-\lambda^*_1z}\{d_1[I_1(\lambda^*_1)-1]-s\lambda^*_1\}\\
&\!\!\!\!\!\!&-(a-1)Be^{-\lambda^*_1z}\left\{d_1\!\int_{\bR}\! J_1(y)ye^{\lambda^*_1y}dy-s\right\}+p_5d_1z^{1/2}e^{-\lambda^*_1z}\\
&\!\!\!\!\!\!&-p_5e^{-\lambda^*_1z}\left\{ d_1\!\int_\bR\! J_1(y)(z-y)^{1/2} e^{\lambda^*_1y}dy-s\lambda^*_1z^{1/2}+\frac{s}{2z^{1/2}} \right\}.
\eeaa
On the other hand, we have, for all $z>z_5\ (>z_0>z_1>0)$,
\beaa
&&r_{1}\underline{\phi}_1(z)[-1-\underline{\phi}_1(z)-k\overline{\phi}_2(z)+a\underline{\phi}_3(z)]\\
&\ge& r_1(a-1)\uphi_1(z)-r_1\ophi_1^2(z)-r_1k\ophi_1(z)\ophi_2(z)-r_1a\ophi_1(z)e^{-\ld_0(z-z_0)}\\
&=&r_1(a-1)^2Bze^{-\lambda^*_1z}-r_1(a-1)p_5z^{1/2}e^{-\lambda^*_1z}\\
& & -r_1(a\!-\!1)^2B^2z^2e^{-2\lambda^*_1z}\!-\!r_1k(a\!-\!1)^2Bze^{-(\lambda^*_1\!+\!\lambda_3)z}\!-\!r_1a(a\!-\!1)Bze^{-\lambda^*_1z\!-\!\ld_0(z\!-\!z_0)}.
\eeaa
Hence we deduce from $g_1(\lambda^*_1)=0$ and \eqref{gi-d1} that $\mL_1(z)\ge e^{-\lambda^*_1z}[p_5A_3(z)-A_4(z)]$ for all $z>z_5$, where
$$\begin{cases}
\displaystyle A_3(z):=d_1\left[z^{1/2}-\int_\bR J_1(y)(z-y)^{1/2}e^{\lambda^*_1y}dy\right]+s\lambda^*_1z^{1/2}-\frac{s}{2z^{1/2}}-r_1(a-1)z^{1/2},\vspace{3pt}\\
A_4(z):=r_1(a-1)^2B^2z^2e^{-\lambda^*_1z}+r_1k(a-1)^2Bze^{-\lambda_3z}+r_1a(a-1)Bze^{-\ld_0(z-z_0)}.\end{cases}$$
Since $A_3$ has the same expression as $A_1$ in~\eqref{defA12}, we can proceed as in the proof of~\eqref{ineqA1}, using $g_1(\lambda^*_1)=0$ and \eqref{gi-d1} again, and we obtain that
$$A_3(z)\ge \frac{d_1}{8(z+S)^{3/2}}\int_\bR J_1(y)y^2e^{\lambda^*_1y}dy\ \hbox{ for }z>z_5.$$
Then, $\mL_1(z)\ge 0$ for $z>z_5$, by the choice of $p_5$ in \eqref{p5}. Hence,~\eqref{l1} holds for all $z\neq z_5$.

Now, as in~\eqref{N1uphi1} in the proof of Lemma~\ref{la:1p}, we have, for every
$$z>z_6=\frac{1}{\mu_3-\lambda_3}\ln\Big(\frac{p_6}{a-1}\Big)>z_0>z_1>0,$$
one has
$$d_2\bN_2[\uphi_2](z)\!+\!s\uphi_2'(z)\!\ge\!(a\!-\!1)e^{-\ld_3 z}\{d_2[I_2(\ld_3)\!-\!1]\!-\!s\ld_3\}\!-\!p_6e^{-\mu_3 z}\{d_2[I_2(\mu_3)\!-\!1]\!-\!s\mu_3\},$$
while
$$\baa{l}
r_2\uphi_2(z)[-1-h\ophi_1(z)-\uphi_2(z)+a\uphi_3(z)]\vspace{3pt}\\
\ \ \ge r_2(a\!-\!1)\big[(a\!-\!1)e^{-\ld_3 z}\!-\!p_6e^{-\mu_3 z}\!-\!h(a\!-\!1)Bze^{-(\lambda^*_1\!+\!\ld_3)z}\!\!-\!(a\!-\!1)e^{-2\ld_3 z}\!\!-\!ae^{-\ld_3 z\!-\!\lambda_0(z\!-\!z_0)}\big].\eaa$$
Hence we obtain from $g_2(\ld_3)=0$ and \eqref{p6} that
\beaa
\mL_2(z)\ge e^{-\mu_3 z}\{-p_6g_2(\mu_3)-r_2(a-1)P_6\}\ge 0\ \hbox{ for all }z>z_6,
\eeaa
Together with the fact that~\eqref{l2} holds trivially for $z<z_6$, we obtain~\eqref{l2} for all $z\neq z_6$.

Finally, \eqref{l3} holds trivially for $z<z_0$. For $z>z_0$, using $z_0>z_1>0$ together with~\eqref{alpha1},~\eqref{a2},~\eqref{defz0bis} and $\ld_0<\rho$, we have
$$\baa{rcl}
\alpha(z)\!-\!b\ophi_1(z)\!-\!b\ophi_2(z)\!-\!\uphi_3(z) & \!\!\!\ge\!\!\! & 1\!-\!\ep e^{-\rho z}\!-\!b(a\!-\!1)Bze^{-\lambda^*_1z}\!-\!b(a\!-\!1)e^{-\lambda_3z}\!-\!1\!+\!e^{-\ld_0(z\!-\!z_0)}\vspace{3pt}\\
& \!\!\!\ge\!\!\! & -\ep e^{-\rho z}\ \ge\ -\ep e^{-\lambda_0 z}.\eaa$$
With~\eqref{ep22}, it follows as in~\eqref{mL3} that $\mL_3(z)\!\ge\!0$ for all $z\!>\!z_0$. Hence \eqref{l3} holds for all $z\!\neq\!z_0$.

Since $0\le\uphi_i\le\ophi_i$ in~$\bR$ for $i=1,2,3$, and each function $\uphi_1$ is nontrivial, the conclusion of Lemma~\ref{la:c1} follows from Lemma~\ref{luslem}, as at the end of the proof of Lemma~\ref{la:1p}.
\end{proof}

Clearly, the solution $\phitri$ given in Lemma~\ref{la:c12} satisfies $\phitri(+\infty)=(0,0,1)$. Together with Proposition~\ref{prop:plus}, Theorem~\ref{th:c12} follows.


\subsection{Waves connecting $(0,0,0)$ and $E_2=(u_p,0,w_p)$: proof of Theorem~\ref{th:c2}}

In this subsection, in addition to~\eqref{SS}, we assume that $d_1=d_2=d_3$, $J_1=J_2=J_3$ in $\R$, that~\eqref{re2} is satisfied, and that $\rho\ge\lambda^{**}_2$, where $\rho>0$ is as in~$($$\alpha 2$$)$ and $\lambda^{**}_2>0$ denotes the unique minimum of the function $R_2$ defined by~\eqref{defs**2} in $(0,\hat\lambda_2)=(0,+\infty)$. We here consider the critical speed
$$s=s^{**}_2,$$
where $s^{**}_2>0$ is defined in~\eqref{defs**2} too. With $\beta_2$ and $G_2$ defined in~\eqref{defbeta2} and~\eqref{defG2} with $s=s^{**}_2$, one has $G_2(\lambda^{**}_2)=0$, and $\lambda^{**}_2$ is the unique positive root of this equation. Similarly to~\eqref{gi-d}, since $R'_2(\lambda^{**}_2)=0$, we also have
\be\label{gg-d}
s=s^{**}_2=d_2\int_{\bR}J_2(y)ye^{\lambda^{**}_2y}dy.
\ee

Recalling $A=a-1-u_p>0$ as in~\eqref{defA} and $w_p+bu_p=1$, we then define
\be\label{def3}\left\{\baa{l}
\ophi_1(z)=\bss u_p+ABze^{-\lambda^{**}_2z}, & z>z_1,\\ a-1, & z\le z_1,\ess\qquad\uphi_1(z)=\bss u_p-u_pBze^{-\lambda^{**}_2 z}, & z>z_1,\\ 0, & z\le z_1,\ess\vspace{3pt}\\
\ophi_2(z)=\bss (a-1)Bze^{-\lambda^{**}_2 z}, & z>z_1,\\ a-1, & z\le z_1,\ess\vspace{3pt}\\
\uphi_2(z)=\bss (a-1)Bze^{-\lambda^{**}_2 z}-q^*z^{1/2}e^{-\lambda^{**}_2 z}, & z>z^*,\\ 0, & z\le z^*,\ess\vspace{3pt}\\
\ophi_3(z)=\bss w_p+bu_pBze^{-\lambda^{**}_2z}, & z>z_1,\\ 1, & z\le z_1,\ess\quad\uphi_3(z)=\bss w_p-w_pBze^{-\lambda^{**}_2 z}, & z>z_1,\\ 0, & z\le z_1,\ess\eaa\right.
\ee
where the parameters $B,q^*,z^*$ are chosen as follows:
\begin{itemize}
\item $B>\lambda^{**}_2 e$ large enough such that $z_1-z_2>S$, where $S$ is as in~\eqref{SS2} and $0<z_2<1/\lambda^{**}_2<z_1$ are such that
$$B=\frac{e^{\lambda^{**}_2z_1}}{z_1}=\frac{e^{\lambda^{**}_2z_2}}{z_2};$$
\item $q^*>0$ large enough such that $z^*:=\{q^*/[(a-1)B]\}^2>z_1$ and
\be\label{qu}
q^*>\frac{8r_2(a-1)(2a-1)B^2\times\displaystyle\mathop{\max}_{z\ge0}\{(z+S)^{3/2}z^2 e^{-\lambda^{**}_2 z}\}}{\displaystyle d_2\int_\bR J_2(y)y^2e^{\lambda^{**}_2 y}dy}.
\ee
\end{itemize}
It is straightforward to check that the functions $\uphi_i$ and $\ophi_i$ are continuous in $\R$, nonnegative, bounded, of class $C^1$ except at finitely many points, and that $\uphi_i\le\ophi_i$ in $\R$, for $i=1,2,3$.

\begin{lemma}\label{la:c2}
Under the above assumptions and notations, there exists a positive solution $\phitri$ of \eqref{TWS2} such that $\uphi_i\le\phi_i\le\ophi_i$ in $\bR$, $i=1,2,3$.
\end{lemma}

\begin{proof}
We first choose $\ep>0$ small enough such that
\be\label{ep-c2}
\ep\le\frac{r_2\beta_2 e}{r_3}.
\ee
Also, up to a translation, condition \eqref{a2} holds for this $\ep$.

Next, we verify \eqref{u1}-\eqref{l3} for the functions defined in \eqref{def3}, for all $z\in\bR$ except
finitely many points. As before, we only need to show~\eqref{u1}-\eqref{l3} for $z>z_1$ or $z>z^*$, respectively: indeed~\eqref{u1}-\eqref{l3} hold trivially for $z<z_1$ or $z<z^*$ respectively since $0\le\uphi_1\le\ophi_1\le a-1$, $0\le\uphi_2\le\ophi_2\le a-1$ and $0\le\uphi_3\le\ophi_3\le1$ in $\R$, together with the precise definitions~\eqref{def3}.

For $z>z_1$, $\ophi_1(z)=u_p+ABz\estar$ and
\be\label{inequphi1}
\ophi_1(z-y)\le u_p+AB(z-y)e^{-\lambda^{**}_2(z-y)}\ \hbox{ for all }y\in[-S,S]
\ee
(indeed, $z-y>z_1-S>z_2$ and if $z-y>z_1$, we have equality in the above inequality, whereas if $z_2<z-y\le z_1$, we have $e^{\lambda^{**}_2(z-y)}/(z-y)\le B$ and $\ophi_1(z-y)=a-1=u_p+A\le u_p+AB(z-y)e^{-\lambda^{**}_2(z-y)}$). Hence
$$\bN_1[\ophi_1](z)\le ABz\estar[I_1(\lambda^{**}_2)-1]-AB\estar\int_\bR J_1(y)ye^{\lambda^{**}_2 y}dy.$$
Also, using $\uphi_2\ge 0$ together with $-1-u_p+aw_p=0$ and $A=a-1-u_p=abu_p$, we have
$$-1-\ophi_1(z)-k\uphi_2(z)+a\ophi_3(z)\le-1-u_p-ABz\estar+aw_p+abu_pBz\estar=0.$$
This implies that, using $G_2(\lambda^{**}_2)=0$ and \eqref{gg-d} together with $d_1=d_2$ and $J_1=J_2$,
$$\baa{rcl}
\mU_1(z) & \le & \displaystyle ABz\estar\{d_1[I_1(\lambda^{**}_2)-1]-s\lambda^{**}_2\}-AB\estar\Big[d_1\int_\bR J_1(y)ye^{\lambda^{**}_2 y}dy-s\Big]\vspace{3pt}\\
& = & -ABr_2\beta_2z\estar<0\eaa$$
for all $z>z_1$. Thus,~\eqref{u1} holds for all $z\neq z_1$.

Similar calculations lead to
\beaa
\mU_2(z)&\le& (a-1)Bz\estar\{d_2[I_2(\lambda^{**}_2)-1]-s\lambda^{**}_2+r_2\beta_2\}\\
&&-(a-1)B\estar\Big[d_2\int_\bR ye^{-\lambda^{**}_2 y}J_2(y)dy-s\Big]\\
& & +[hu_p-(a-1)+abu_p]r_2(a-1)B^2z^2e^{-2\lambda^{**}_2z}<0
\eeaa
for all $z>z_1$, using $G_2(\lambda^{**}_2)=0$, \eqref{gg-d} and $hu_p-(a-1)+abu_p<(ab+1)u_p-(a-1)=0$. Therefore,~\eqref{u2} holds for all $z\neq z_1$.

As for~\eqref{u3}, since $\uphi_2\ge 0$ and $1-bu_p-w_p=0$, we get that
$$\alpha(z)-b\uphi_1(z)-b\uphi_2(z)-\ophi_3(z)\le 1-bu_p+bu_pBz\estar-w_p-bu_pBz\estar=0$$
for every $z>z_1$, and, together with $G_2(\lambda^{**}_2)=0$ and \eqref{gg-d}, it follows that
\beaa
\mU_3(z)&\le& bu_pBz\estar\{d_3[I_3(\lambda^{**}_2)-1]-s\lambda^{**}_2\}-bu_pB\estar\Big[d_3\int_\bR J_3(y)ye^{\lambda^{**}_2 y}dy-s\Big]\\
&=& -bu_pBr_2\beta_2z\estar<0,
\eeaa
since $J_3=J_2$ and $d_3=d_2$. As a consequence,~\eqref{u3} holds for all $z\neq z_1$.

For $\mL_1$ and~\eqref{l1}, one has, for every $z>z_1$, $\uphi_1(z)=u_p-u_pBz\estar$ and
$$\uphi_1(z-y)\ge u_p-u_pB(z-y)e^{-\lambda^{**}_2(z-y)}\ \hbox{ for all }y\in[-S,S]$$
with the same arguments as for the proof of~\eqref{inequphi1}, hence
$$\baa{rcl}
d_1\bN_1[\uphi_1](z)+s\uphi_1(z) & \ge & \displaystyle-u_pBz\estar\{d_1[I_1(\lambda^{**}_2)-1]-s\lambda^{**}_2\}\vspace{3pt}\\
& & \displaystyle+u_pB\estar\Big\{d_1\int_\bR J_1(y)ye^{\lambda^{**}_2 y}dy-s\Big\}\vspace{3pt}\\
&= & u_pBr_2\beta_2z\estar\eaa$$
due to $G_2(\lambda^{**}_2)=0$,~\eqref{gg-d}, $d_1=d_2$, and $J_1=J_2$. On the other hand, using $u_p-aw_p=-1$, there holds
\beaa
-1-\uphi_1(z)-k\ophi_2(z)+a\uphi_3(z)&=&(-1-u_p+aw_p)+Bz\estar[u_p-k(a-1)-aw_p]\\
&=& -[1+k(a-1)]Bz\estar
\eeaa
for all $z>z_1$, hence
\beaa
\mL_1(z)&\ge& u_pBr_2\beta_2z\estar-r_1(u_p-u_pBz\estar)[1+k(a-1)]Bz\estar\\
&\ge& u_pBz\estar\{r_2\beta_2-r_1[1+k(a-1)]\}\ge 0,
\eeaa
using $r_2\beta_2\ge r_1[1+k(a-1)]$ by~\eqref{re2}. Thus,~\eqref{l1} holds for all $z\neq z_1$.

To show~\eqref{l2} for $z>z^*\ (>z_1)$, we can proceed for instance as in~\eqref{N1uphi1bis} in the proof of Lemma~\ref{la:c1} to reach
\beaa
d_2\bN_2[\uphi_2](z)+s\uphi_2'(z)&\ge&(a-1)Bze^{-\lambda^{**}_2 z}\left\{d_2\left[\int_\bR J_2(y)e^{\lambda^{**}_2 y}dy-1\right]-s\lambda^{**}_2\right\}\\
&&-(a-1)Be^{-\lambda^{**}_2 z}\left\{ d_2\int_{\bR} J_2(y)ye^{\lambda^{**}_2 y}dy-s\right\}+q^*d_2e^{-\lambda^{**}_2 z}z^{1/2}\\
&&-q^*e^{-\lambda^{**}_2 z}\left\{ d_2\int_\bR J_2(y)(z-y)^{1/2} e^{\lambda^{**}_2 y}dy-s\lambda^{**}_2 z^{1/2}+\frac{s}{2z^{1/2}}\right\},
\eeaa
while
\beaa
&&r_{2}\uphi_2(z)[-1-h\ophi_1(z)-\uphi_2(z)+a\underline{\phi}_3(z)]\\
&=& r_2\uphi_2(z)\big[-1-hu_p-hABz\estar-\uphi_2(z)+aw_p-aw_pBz\estar\big]\\
&\ge&r_2\beta_2\uphi_2(z)-r_2\ophi_2(z)(hA+aw_p)Bz\estar-r_2\ophi_2^2(z)\\
&\ge&r_2\beta_2\uphi_2(z)-r_2(a-1)(2a-1)B^2z^2 e^{-2\lambda^{**}_2 z},
\eeaa
using $hA+aw_p+(a-1)<A+aw_p+(a-1)=2a-1$. It follows from $G_2(\lambda^{**}_2)=0$ and~\eqref{gg-d} that $\mL_2(z)\ge \estar[q^*A_5(z)-A_6(z)]$ for $z>z^*$, where
$$\left\{\baa{l}
\displaystyle A_5(z):=d_2\left[z^{1/2}-\int_\bR J_2(y)(z-y)^{1/2}e^{\lambda^{**}_2 y}dy\right]+s\lambda^{**}_2z^{1/2}-\frac{s}{2z^{1/2}}-r_2\beta_2z^{1/2},\vspace{3pt}\\
\displaystyle A_6(z):=r_2(a-1)(2a-1)B^2z^2 e^{-\lambda^{**}_2 z}.\eaa\right.$$
Then, using $G_2(\lambda^{**}_2)=0$ together with~\eqref{gg-d} and~\eqref{qu}, the same arguments as those in the proofs of Lemmas~\ref{la:c1} and~\ref{la:c12} give $\mL_2(z)\ge 0$ for all $z>z^*$. Therefore,~\eqref{l2} holds for all $z\neq z^*$.

Finally, let us consider $\mL_3$. For $z>z_1$, $\uphi_3(z)=w_p-w_pBz\estar$ and
$$\uphi_3(z-y)\ge w_p-w_pB(z-y)e^{-\lambda^{**}_2(z-y)}\ \hbox{ for all }y\in[-S,S],$$
as in the proof of~\eqref{inequphi1} for instance. Then, using $d_3=d_2$ and $J_3=J_2$, together with $G_2(\lambda^{**}_2)=0$ and~\eqref{gg-d}, one infers that
\beaa
d_3\bN_3[\uphi_3](z)\!+\!s\uphi_3'(z)&\!\!\!\!\!\ge\!\!\!\!\!& -w_pBz\estar\{d_3[I_3(\lambda^{**}_2)\!-\!1]\!-\!s\lambda^{**}_2\}\!+\!w_pB\estar\!\Big\{\!d_3\!\!\int_\bR\!\!\! J_3(y)ye^{\lambda^{**}_2 y}dy\!-\!s\!\Big\}\\
&\!\!\!\!\!=\!\!\!\!\!& \ w_pBr_2\beta_2z\estar
\eeaa
for all $z>z_1$. On the other hand, using \eqref{a2}, $bu_p+w_p=1$ and $2b(a-1)<1$, we obtain
\beaa
&&\alpha(z)-b\ophi_1(z)-b\ophi_2(z)-\uphi_3(z)\\
&\ge& 1-\ep e^{-\rho z}-bu_p-bABz\estar-b(a-1)Bz\estar-w_p+w_pBz\estar\ \ge\ -\ep e^{-\rho z}
\eeaa
for all $z>z_1$, whence
$$\mL_3(z)\ge w_pBr_2\beta_2z\estar-r_3\uphi_3(z)\ep e^{-\rho z}\ge w_p\estar(Br_2\beta_2z_1-r_3\ep)\ge 0,$$
using $\rho\ge\lambda^{**}_2$ and $Br_2\beta_2z_1\ge\lambda^{**}_2er_2\beta_2/\lambda^{**}_2=r_2\beta_2 e\ge r_3\ep$ (because $z_1>1/\lambda^{**}_2$, $B>\lambda^{**}_2 e$ and because of~\eqref{ep-c2}). As a consequence,~\eqref{l3} holds for all $z\neq z_1$.

Since $0\le\uphi_i\le\ophi_i$ in~$\bR$ for $i=1,2,3$, and each function $\uphi_i$ is nontrivial, the conclusion of Lemma~\ref{la:c2} follows from Lemma~\ref{luslem}, as at the end of the proof of Lemma~\ref{la:1p}.
\end{proof}

Clearly, the solution $\phitri$ given in Lemma~\ref{la:c2} satisfies $\phitri(+\infty)=(u_p,0,w_p)$. Together with Proposition~\ref{prop:plus}, Theorem~\ref{th:c2} follows.


%
%
%
%
%
%
%


\end{document}